\documentclass{article}

\usepackage{cite}
\usepackage{amssymb}
\usepackage{amsmath}
\usepackage{amsfonts}
\usepackage{amsthm}
\usepackage{fancyhdr}
\usepackage{lastpage}
\usepackage[explicit]{titlesec}
\usepackage{titletoc}
\usepackage{booktabs}
\usepackage[all]{xy}
\usepackage[colorlinks=true,linkcolor=blue]{hyperref}
\usepackage[mathscr]{euscript}

\allowdisplaybreaks

\usepackage{cancel}

\usepackage{url}
\DeclareMathOperator\ch{ch}
\DeclareMathOperator\cch{cch}
\DeclareMathOperator\ex{ex}

\usepackage{hyperref}

\newtheorem{proposition}{Proposition}[section]

\newtheorem{theorem}{Theorem}[section]
\newtheorem{corollary}{Corollary}[section]
\newtheorem{definition}{Definition}[section]

\newtheorem{lemma}{Lemma}[section]
\newtheorem{claim}{Claim}[section]
\newtheorem{question}{Question}[section]

\everymath{\displaystyle}

\begin{document}

\title{Topology and convexity in the space of actions modulo weak equivalence}

\author{Peter Burton\footnote{Research partially supported by NSF grant DMS-0968710}}

\date{\today}

\maketitle

\begin{abstract} We analyse the structure of the quotient $\mathrm{A}_\sim(\Gamma,X,\mu)$ of the space of measure-preserving actions of a countable discrete group by the relation of weak equivalence. This space carries a natural operation of convex combination. We introduce a variant of an abstract construction of Fritz which encapsulates the convex combination operation on $\mathrm{A}_\sim(\Gamma,X,\mu)$. This formalism allows us to define the geometric notion of an extreme point. We also discuss a topology on $\mathrm{A}_\sim(\Gamma,X,\mu)$ due to Abert and Elek in which it is Polish and compact, and show that this topology is equivalent others defined in the literature. We show that the convex structure of $\mathrm{A}_\sim(\Gamma,X,\mu)$ is compatible with the topology, and as a consequence deduce that $\mathrm{A}_\sim(\Gamma,X,\mu)$ is path connected. Using ideas of Tucker-Drob we are able to give a complete description of the topological and convex structure of $\mathrm{A}_\sim(\Gamma,X,\mu)$ for amenable $\Gamma$ by identifying it with the simplex of invariant random subgroups. In particular we conclude that $\mathrm{A}_\sim(\Gamma,X,\mu)$ can be represented as a compact convex subset of a Banach space if and only if $\Gamma$ is amenable. In the case of general $\Gamma$ we prove a Krein-Milman type theorem asserting that finite convex combinations of the extreme points of $\mathrm{A}_\sim(\Gamma,X,\mu)$ are dense in this space. We also consider the space $\mathrm{A}_{\sim_s}(\Gamma,X,\mu)$ of stable weak equivalence classes and show that it can always be represented as a compact convex subset of a Banach space. In the case of a free group $\mathbb{F}_N$, we show that if one restricts to the compact convex set $\mathrm{FR}_{\sim_s}(\mathbb{F}_N,X,\mu) \subseteq \mathrm{A}_{\sim_s}(\mathbb{F}_N,X,\mu)$ consisting of the stable weak equivalence classes of free actions, then the extreme points are dense in $\mathrm{FR}_{\sim_s}(\mathbb{F}_N,X,\mu)$.  \end{abstract}

\section{Introduction.}

By a probability space we mean a standard Borel space $Y$ with a Borel probability measure $\nu$. If $\nu$ is nonatomic, we say the pair $(Y,\nu)$ is a standard probability space. If $\nu$ is nonatomic then $Y$ must be uncountable and thus by Theorem 17.41 in \cite{K95} every standard probability space is isomorphic to the unit interval with Lebesgue measure. Let $\Gamma$ be a countable discrete group. By a measure-preserving action of $\Gamma$ on $(Y,\nu)$ we mean a Borel action $a: \Gamma \times Y \to Y$ which preserves the measure $\nu$. We write $\Gamma \curvearrowright^a (Y,\nu)$. In accordance with the standard conventions of ergodic theory, we identify two actions which agree almost everywhere. Thus a measure-preserving action of $\Gamma$ on $(Y,\nu)$ is equivalently a homomorphism from $\Gamma$ into the group $\mathrm{Aut}(Y,\nu)$ of measure-preserving automorphisms of $(Y,\nu)$, where again two such automorphisms are identified if they agree almost everywhere.\\
\\
We fix a standard probability space $(X,\mu)$ throughout the remainder of the paper. As in \cite{K} we can define the Polish space $\mathrm{A}(\Gamma,X,\mu)$ of measure-preserving actions of $\Gamma$. Kechris defines the following relation of weak containment among measure-preserving actions of $\Gamma$, by analogy with the standard notion of weak containment for representations.

\begin{definition} \cite{K} If $\Gamma \curvearrowright^a (X,\mu)$ and $\Gamma \curvearrowright^b (Y,\nu)$ are measure-preserving actions of $\Gamma$ on probability spaces, we say $a$ is \textbf{\textbf{weakly contained}} in $b$ and write $a \prec b$ if for any finite sequence $A_1,\ldots,A_n$ of measurable subsets of $X$, finite $F \subseteq \Gamma$ and $\epsilon > 0$ there exist measurable subsets $B_1,\ldots,B_n$ of $Y$ such that for all $\gamma \in F$ and all $i,j \leq n$ we have \[ | \mu( \gamma^a A_i \cap A_j) - \nu( \gamma^b B_i \cap B_j)| < \epsilon. \] We say $a$ is \textbf{\textbf{weakly equivalent}} to $b$ and write $a \sim b$ if $a \prec b$ and $b \prec a$. \end{definition}

We may assume in this definition that $A_1,\ldots,A_n$ form a partition of $X$. Note that we do not require $(X,\mu)$ and $(Y,\nu)$ to be standard, that is to say we include the case where they might be countable. The relation of weak containment is $G_\delta$, so the quotient $\mathrm{A}_\sim(\Gamma,X,\mu)$ of $\mathrm{A}(\Gamma,X,\mu)$ by weak equivalence is well-behaved. \\
\\
We also consider a generalization of weak containment, due to Tucker-Drob. For probability spaces $(Y_i,\nu_i), 1 \leq i \leq m$ and positive real numbers $\alpha_i, 1 \leq i \leq m$ with $\sum_{i=1}^m \alpha_i = 1$ we let $\bigsqcup_{i=1}^m \alpha_i Y_i$ be the probability space formed by endowing the disjoint union of the $Y_i$ with the measure $\sum_{i=1}^m \alpha_i \nu_i$ given by $\left( \sum_{i=1}^m \alpha_i \nu_i \right)(A) = \sum_{i=1}^m \alpha_i \nu_i(A \cap Y_i)$. If $\Gamma \curvearrowright^{a_i} (Y_i,\nu_i)$ are measure-preserving actions, then $\sum_{i=1}^m \alpha_i a_i$ is the action on $\bigsqcup_{i=1}^m \alpha_i Y_i$ given by letting $\Gamma$ act like $a_i$ on $Y_i$.

\begin{definition} \cite{RTD} If $\Gamma \curvearrowright^a (X,\mu)$ and $\Gamma \curvearrowright^b (Y,\nu)$ are measure-preserving actions, we say $a$ is \textbf{\textbf{stably weakly contained}} in $b$ if for all $A_1,\ldots,A_k \in \mathrm{MALG}_\mu$, all finite $F \subseteq \Gamma$ and all $\epsilon > 0$ there exist $\alpha_1, \ldots,\alpha_m$ such that $\sum_{i=1}^m \alpha_1 = 1$ and sets $B_1,\ldots,B_k \subseteq \bigsqcup_{i=1}^m \alpha_i Y_i$ such that \[ \left \vert \mu(\gamma^a A_i \cap A_j) - \sum_{i=1}^m \alpha_i \nu \left( \gamma^{\sum_{i=1}^m \alpha_i b} B_i \cap B_j \right) \right \vert < \epsilon.\] We write $a \prec_s b$ if $a$ is stably weakly contained in $b$ and $a \sim_s b$ for $a \prec_s b$ and $b \prec_s a$. \end{definition}

When we wish to distinguish between and action and its equivalence class, we write $[a]$ for the weak equivalence class of $a$ and $[a]_s$ for the stable weak equivalence class. The quotient of $\mathrm{A}(\Gamma,X,\mu)$ by the relation of stable weak containment is denoted $\mathrm{A}_{\sim_s}(\Gamma,X,\mu)$. The goal of this paper is to analyze the topological and geometric structure of $\mathrm{A}_\sim(\Gamma,X,\mu)$ and $\mathrm{A}_{\sim_s}(\Gamma,X,\mu)$ .\\
\\
More specifically, unlike $\mathrm{A}(\Gamma,X,\mu)$, the spaces $\mathrm{A}_\sim(\Gamma,X,\mu)$ and $\mathrm{A}_{\sim_s}(\Gamma,X,\mu)$ carry a well-defined operation of convex combination. This is inherited from the operation of endowing the disjoint union of two probability spaces with a convex combination of their respective measures. In Section 2 we introduce a variation of a construction of Fritz \cite{Fr} which abstracts the idea of convex combinations. Fritz's objects are referred to as `convex spaces'; we weaken the definition in order to encompass the convex structure on $\mathrm{A}_\sim(\Gamma,X,\mu)$ obtaining the notion of `weak convex space'. We show that this relates naturally to other ideas of convexity, define a notion of convex function and generalize the important geometric notions of `convex hull', `extreme point' and `face' from the classical situation of vector spaces to this abstract framework. We also define `topological weak convex spaces' as weak convex structures which are appropriately compatible with an underlying topology.\\
\\
In Section 3 we consider methods of topologizing $\mathrm{A}_\sim(\Gamma,X,\mu)$. The first topology defined on this space was in \cite{AbEl}, and a second formulation was given in \cite{RTD}. These are equivalent, Polish, compact and finer than the quotient of the weak topology on $\mathrm{A}(\Gamma,X,\mu)$. We discuss a third topology, implicit in \cite{AbEl} and pointed out to us by Kechris. This is shown to be equivalent to the previous two. We also consider a natural topology on $\mathrm{A}_{\sim_s}(\Gamma,X,\mu)$.\\
\\
In Section 4 we describe how to endow $\mathrm{A}_\sim(\Gamma,X,\mu)$ with the structure of a weak convex space and show that it is in fact a topological weak convex space. Furthermore, we show that the metric giving $\mathrm{A}_\sim(\Gamma,X,\mu)$ its Polish topology is compatible with the convex structure in the sense that the distance function to any compact convex set is a convex function. \\
\\
In Section 5 we analyze the structure of $\mathrm{A}_\sim(\Gamma,X,\mu)$ for amenable $\Gamma$. The main tool is the following idea. Let $\mathrm{Sub}(\Gamma)$ be the space of subgroups of $\Gamma$, regarded as a subspace of $\{0,1\}^\Gamma$ with the product topology. $\mathrm{Sub}(\Gamma)$ is then a compact metric space on which $\Gamma$ acts by conjugation.

\begin{definition} An \textbf{\textbf{invariant random subgroup}} of $\Gamma$ is a conjugation-invariant Borel probability measure on $\mathrm{Sub}(\Gamma)$. \end{definition}

Invariant random subgroups have been studied in numerous recent papers, including \cite{AGV}, \cite{Bow}, \cite{BGK} and \cite{EisGlas}. If $\Gamma \curvearrowright^a (X,\mu)$ is a measure-preserving action, then the pushforward measure $(\mathrm{stab}_a)_* \mu$ is an invariant random subgroup of $\Gamma$ called the type of $a$. We extend ideas of Tucker-Drob from \cite{RTD} to show the following.

\begin{theorem}\label{thm1} If $\Gamma$ is amenable, then $\mathrm{A}_\sim(\Gamma,X,\mu)$ is isomorphic to $\mathrm{IRS}(\Gamma)$ as a topological convex space. In particular, if $\Gamma$ is amenable then $\mathrm{A}_\sim(\Gamma,X,\mu)$ is isomorphic to a compact convex subset of a Banach space. \end{theorem}

In Section 6 we consider the structure of $\mathrm{A}_\sim(\Gamma,X,\mu)$ for general $\Gamma$. If $\Gamma$ is nonamenable, the existence of strongly ergodic actions of $\Gamma$ implies that the convex structure on this space has the pathology that the convex combination of a point $x$ with itself might be different from $x$. This is why we need to consider weak convex spaces instead of just convex spaces. The main result of this section is the following Krein-Milman type theorem.

\begin{theorem} \label{thm2} $\mathrm{A}_\sim(\Gamma,X,\mu)$ is equal to the closed convex hull of its extreme points. In other words, finite convex combinations of the extreme points of $\mathrm{A}_\sim(\Gamma,X,\mu)$ are dense in $\mathrm{A}_\sim(\Gamma,X,\mu)$. \end{theorem}

Given this result, it seems interesting to describe the extreme points of $\mathrm{A}_\sim(\Gamma,X,\mu)$. In the amenable case, the identification with $\mathrm{IRS}(\Gamma)$ provides a complete such description, since the extreme points of $\mathrm{IRS}(\Gamma)$ are known to be the ergodic measures and consequently the extreme points of $\mathrm{A}_\sim(\Gamma,X,\mu)$ for amenable $\Gamma$ are exactly those actions with ergodic type. In the nonamenable case this description does not suffice. It is clear that any strongly ergodic action is an extreme point. We are able to show the following.

\begin{theorem} \label{thm4} Suppose $[a] \in \mathrm{A}_\sim(\Gamma,X,\mu)$ is an extreme point. Let $a = \int_Z a_z d \eta(z)$ be the ergodic decomposition of $a$. Then there is a measure-preserving action $b$ of $\Gamma$ such that for $\eta$-almost all $z \in Z$ we have $[a_z] = [b]$. \end{theorem}

Let $\mathrm{FR}_{\sim}(\Gamma,X,\mu)$ denote the subspace of $\mathrm{A}_\sim(\Gamma,X,\mu)$ consisting of the weak equivalence classes of free actions. We prove:

\begin{theorem} \label{thm5} Let $\mathbb{F}_N$ be a free group of finite or countably infinite rank. Then the weak equivalence classes containing a free ergodic action are dense in $\mathrm{FR}_{\sim}(\mathbb{F}_N,X,\mu)$.\end{theorem}

In Section 7 we use a characterization of convex subsets of Banach spaces from \cite{CFr} to show the following.

\begin{theorem}\label{thm3} For any $\Gamma$, the space $\mathrm{A}_{\sim_s}(\Gamma,X,\mu)$ is isomorphic to a compact convex subset of a Banach space. \end{theorem}

We characterize the extreme points of $A_{\sim_s}(\Gamma,X,\mu)$ as precisely those stable weak equivalence classes which contain an ergodic action. This result was obtained by Tucker-Drob and Bowen independently of the author. Tucker-Drob and Bowen have also shown that $\mathrm{A}_{\sim_s}(\Gamma,X,\mu)$ is a simplex, and the set $\mathrm{FR}_{\sim_s}(\Gamma,X,\mu)$ of stable weak equivalence classes of free actions is a subsimplex. Recall that a Poulsen simplex is a simplex such that the extreme points are dense. Thus from Theorem \ref{thm5} we have:

 \begin{corollary}\label{cor1} Let $\mathbb{F}_N$ be a free group of finite or countably infinite rank. Then $\mathrm{FR}_{\sim_s}(\mathbb{F}_N,X,\mu)$ is a Poulsen simplex. \end{corollary}

\subsection*{Acknowledgements}

We would like to thank Alexander Kechris for introducing us to this topic and for many helpful discussions. We also thank Robin Tucker-Drob for informing us of his result with Bowen that the space of stable weak equivalence classes forms a simplex, and for raising the question of when it forms a Poulsen simplex.

\section{Weak convex spaces.}

We first describe the formalism realized by $\mathrm{A}_\sim(\Gamma,X,\mu)$.

\subsection{Convex spaces and weak convex spaces.}

Convex spaces were introduced in \cite{Fr} and further developed in \cite{CFr} as an abstract setting to study the notion of convex combination.

\begin{definition}\label{def1} \cite{Fr} A \textbf{\textbf{convex space}} is a set $X$ together with a family $\mathcal{V}$ of binary operations $cc_{t}$ for each $t \in [0,1]$ such that for all $x,y,z \in X$ and all $s,t \in [0,1]$ \begin{description} \item{(1)} $cc_0(x,y) = x$, \item{(2)} $cc_t(x,x) = x$, \item{(3)} $cc_t(x,y) = cc_{1-t}(y,x)$, \item{(4)} $cc_{t}(cc_s(x,y),z) = cc_{st}\left (x,cc_{\frac{t(1-s)}{1-st}}(y,z)\right)$. \end{description} \end{definition}

We will usually write $tx +_{\mathcal{V}} (1-t)y$ for $cc_t(x,y)$, omitting the subscript $\mathcal{V}$ when the convex structure being considered is clear. Note that $\mathrm{(4)}$ allows us to unambiguously define $\sum_{i=1}^n \lambda_i x_i$ for $(x_i)_{i=1}^n \subseteq X$ and $(\lambda_i)_{i=1}^n \subseteq [0,1]$ such that $\sum_{i=1}^n \lambda_i = 1$. We will need to weaken the definition of a convex space to cover the situation where a convex combination of a point $x$ with itself could be different from $x$.

\begin{definition} An \textbf{\textbf{weak convex space}} is a set $X$ with a family $cc_{t}$ of binary operations for $t \in [0,1]$ satisfying $\mathrm{(1)},\mathrm{(3)}$ and $\mathrm{(4)}$ of Definition \ref{def1}. \end{definition}

\begin{definition} A \textbf{\textbf{topological (weak) convex space}} is a topological space $X$ carrying a (weak) convex structure such that the ternary operation $cc: [0,1] \times X^2 \to X$ given by $cc(t,x,y) = cc_t(x,y)$ is continuous. \end{definition}

\subsection{Extreme points and faces.}

We can define extreme points in a weak convex space in exactly the same way as in a vector space.

\begin{definition} If $A$ is a convex set in a weak convex space, we say $x \in A$ is an \textbf{\textbf{extreme point}} if $x = ty + (1-t)z$ for $0 < t<1$ and some $y,z \in A$ implies $y = z = x$. Write $\mathrm{ex}(A)$ for the set of extreme points of $A$. If $A$ is a compact convex subset of a topological weak convex space, we say a \textbf{\textbf{face}} of $A$ is a nonempty closed subset $F \subseteq A$ such that if $x,y \in A$, $0 <t <1$ and $tx + (1-t)y \in F$ then $x,y \in F$. \end{definition}

\section{Topology on $\mathrm{A}_\sim(\Gamma,X,\mu)$.} \label{sec3}

Let $\Gamma$ be a countable group and $\mathrm{A}_\sim(\Gamma,X,\mu)$ be its space of actions modulo weak equivalence. We consider a metric on $\mathrm{A}_\sim(\Gamma,X,\mu)$ which is implicit in \cite{AbEl}.\\
\\
Fix an enumeration $(\gamma_i)_{i=0}^\infty$ of $\Gamma$. If $\mathcal{A} = \{A_1,\ldots,A_k\}$ is a partition of $X$ into $k$ pieces, $a \in \mathrm{A}(\Gamma,X,\mu)$ and $n \in \mathbb{N}$, let $M^{\mathcal{A}}_{n,k}(a) \in [0,1]^{n \times k \times k}$ be the point whose $p,q,r$ coordinate is $\mu(\gamma_p^a A_q \cap A_r)$, where $p \leq n$ and $q,r \leq k$. Let $C_{n,k}(a) = \overline{ \{ M_{n,k}^\mathcal{A}(a): \mathcal{A} \mbox{ is a partition of } X \mbox{ into } k \mbox{ pieces.} \} }$ Then we can define a pseudometric $d$ on $\mathrm{A}(\Gamma,X,\mu)$ by the formula \[ d(a,b) = \sum_{n,k=1}^\infty \frac{1}{2^{n+k}} d_H( C_{n,k}(a), C_{n,k}(b) ) \] where $d_H$ is the Hausdorff distance in the hyperspace of compact subsets of $[0,1]^{n \times k \times k}$. It is easy to see that $a \sim b$ if and only if $d(a,b) = 0$, so $d$ descends to a metric on $\mathrm{A}_\sim(\Gamma,X,\mu)$, which we also denote by $d$. Let $\tau_1$ be the topology induced by $d$. We note that this definition extends to actions on countable spaces. We will write $A_\sim^*(\Gamma)$ for the space of all actions of $\Gamma$ on probability spaces. \\
\\
We now describe a different construction of the topology on $\mathrm{A}_\sim(\Gamma,X,\mu)$ due to Tucker-Drob \cite{RTD} in order to show it agrees with the one we have just introduced. (Tucker-Drob shows in \cite{RTD} that his formulation agrees with the one from \cite{AbEl}). \\
\\
Let $S$ be a compact Polish space, and consider $S^\Gamma$, which is also a compact Polish space. $\Gamma$ acts on $S^\Gamma$ by the shift action $s$ given by $(\gamma^s f)(\delta) = f(\gamma^{-1} \delta)$. Let $M_s(S^\Gamma)$ be the compact Polish space of shift-invariant probability measures on $S^\Gamma$ and let $\mathcal{K}_S = \mathcal{K}(M_s(S^\Gamma))$ be the hyperspace of compact subsets of $M_s(S^\Gamma)$ equipped with the Hausdorff topology. Then $\mathcal{K}_S$ is again compact and Polish. Now consider a $S$-valued random variable $\phi \in L(X,\mu,S)$ on $X$, that is to say a measurable map $\phi: X \to S$. For each a measure-preserving action $a \in \mathrm{A}(\Gamma,X,\mu)$ we get a map $\Phi_S^{\phi,a}: X \to S^\Gamma$ by letting $\Phi_S^{\phi,a}(x)(\gamma) = \phi( (\gamma^{-1})^a x)$ and consequently a shift-invariant measure $(\Phi_S^{\phi,a})_* \mu$ on $S^\Gamma$. Then define a subset $E(a,S)$ of $M_s(S^\Gamma)$ by \[E(a,S) = \{ (\Phi_S^{\phi,a})_* \mu: \phi: X \to S \mbox{ is measurable} \}.\]  Let $\Phi_S: \mathrm{A}(\Gamma,X,\mu) \to \mathcal{K}_S$ be given by $\Phi_S(a) = \overline{E(a,S)}$. When $S = K$ is the Cantor set, we omit the subscript $S$ on the notations just introduced. By Proposition 3.5 in \cite{RTD}, we have $a \sim b$ if and only if $\Phi(a) = \Phi(b)$ so we can consider the initial topology on $\mathrm{A}_\sim(\Gamma,X,\mu)$ induced by $\Phi$. Call this $\tau_2$. We now work towards showing $\tau_1$ agrees with $\tau_2$. There will be a series of preliminary steps. This entire argument can be regarded as a `perturbed' version of Proposition 3.5 in \cite{RTD}.\\
\\
We first fix a compatible metric on $M_s(K^\Gamma)$. Let $\mathcal{A}_K$ be the collection of clopen subsets of $K^\Gamma$ of the form $\pi_F^{-1} \left ( \prod_{\gamma \in F} A_\gamma \right)$ where $A_\gamma \subseteq K$ an element of some fixed countable clopen basis for $K$, $F \subseteq \Gamma$ is finite and $\pi: K^\Gamma \to K^F$ is the projection onto the $F$-coordinates. Since the elements of $\mathcal{A}_K$ generate the Borel $\sigma$-algebra of $K^\Gamma$, for $(\nu_n)_{n=1}^\infty \subseteq M_s(K^\Gamma)$ we have $\nu_n \to \nu$ in $M_s(K^\Gamma)$ if and only if $\nu_n(A) \to \nu(A)$ for every $A \in \mathcal{A}_K$. So, enumerating the elements of $\mathcal{A}_K$ as $(A^K_i)_{i=1}^\infty$, $\delta_K$ given by \[ \delta_K(\nu,\rho) = \sum_{i=1}^\infty \frac{1}{2^i} | \nu(A^K_i) - \rho(A^K_i)| \] is a compatible metric on $M_s(K^\Gamma)$. 

\begin{lemma}\label{lem5} For any $\epsilon > 0$ there exists $k \in \mathbb{N}$ such that every $a$ and every $\phi \in L(X,\mu,K)$ there is $\psi \in L(X,\mu,K)$ with $\delta_K((\Phi^{\phi,a})_* \mu, (\Phi^{\psi,a})_* \mu) < \epsilon$ such that the range of $\psi$ has size $\leq k$. Note that $k$ depends only on $\epsilon$, not on $a$ or $\phi$. \end{lemma}

\begin{proof} Fix $\epsilon$. Choose $N$ large enough that $\sum_{i=N}^\infty \frac{1}{2^i} < \epsilon$. For each $i \leq N$, write $A_i = \pi_{F_i}^{-1} \left( \prod_{\gamma \in F_i} A^i_\gamma \right)$ for $A^i_\gamma \subseteq K$ clopen and $F_i \subseteq \Gamma$ finite. We have for all $a \in \mathrm{A}(\Gamma,X,\mu)$ and $\phi,\psi \in L(X,\mu,K)$, \begin{align} | \Phi^{\phi,a}(A_i) - \Phi^{\psi,a}(A_i)| & = \left \vert  \Phi^{\phi,a}\left ( \pi_{F_i}^{-1} \left( \prod_{\gamma \in F_i} A^i_\gamma \right) \right)- \Phi^{\psi,a}\left( \pi_{F_i}^{-1} \left( \prod_{\gamma \in F_i} A^i_\gamma \right) \right) \right \vert \nonumber \\ & = |\mu( \{x: \Phi^{\phi,a}(x)(\gamma) \in A^i_\gamma \mbox{ for all } \gamma \in F_i \} ) - \mu( \{x: \Phi^{\psi,a}(x)(\gamma) \in A^i_\gamma \mbox{ for all } \gamma \in F_i \} ) | \nonumber \\ & = | \mu(\{x: \phi((\gamma^{-1})^a x) \in A^i_\gamma \mbox{ for all } \gamma \in F_i\}) -\mu(\{x: \psi((\gamma^{-1})^a x) \in A^i_\gamma \mbox{ for all } \gamma \in F_i \})| \nonumber \\ & =\left \vert \mu \left ( \bigcap_{\gamma \in F_i} \gamma^a \phi^{-1}(A^i_\gamma) \right) - \left( \bigcap_{\gamma \in F_i} \gamma^a \psi^{-1}(A^i_\gamma) \right) \right \vert.  \end{align}

Now, fix $\phi \in L(X,\mu,K)$. Let $(B_j)_{j=1}^k$ be the finite partition of $K$ given by the atoms of the Boolean algebra generated by $(A^i_\gamma)_{i \leq N, \gamma \in F_i}$. Note that $k$ depends only on $\epsilon$. For each $j \leq k$, let $y_j$ be any point in $B_j$. Define a map $\psi: X \to K$ by letting $\psi(x) = y_j$ for the unique $j$ such that $x \in \phi^{-1}(B_j)$. Then $\psi^{-1}(B_j) = \phi^{-1}(B_j)$ for each $j$, and hence $\phi^{-1}(A^i_\gamma) = \psi^{-1}(A^i_\gamma)$ for each $i \leq N$ and $\gamma \in F_i$. Therefore the value of the expression $(1) $ is $0$ and $\delta_K( (\Phi^{\phi,a})_* \mu, (\Phi^{\psi,a})_* \mu ) < \epsilon.$ \end{proof}

\begin{lemma}\label{lem6} If $\overline{ E(a_n,L)} \to \overline{ E(a,L)}$ in $\mathcal{K}(M_s(L^\Gamma))$ for every finite set $L$ then $\overline{ E(a_n,K)} \to \overline{E(a,K)}$ in $\mathcal{K}(M_s(K^\Gamma))$. \end{lemma}

\begin{proof} Fix $\epsilon > 0$ in order to show that eventually $d_{\mathcal{K}}\left(\overline{E(a_n,K)},\overline{E(a,K)} \right) < \epsilon$, where $d_{\mathcal{K}}$ is the Hausdorff distance in $\mathcal{K}(M_s(K^\Gamma))$ constructed from $\delta_K$. For $k \in \mathbb{N}$ and $b \in \mathrm{A}(\Gamma,X,\mu)$ let \[E_k(b,K) = \{ (\Phi^{\phi,a})_* \mu: \phi: X \to K \mbox{ is measurable and the range of } \phi \mbox{ has size } \leq k \}. \] By Lemma \ref{lem5} we can choose $k \in \mathbb{N}$ such that $E(b,K) \subseteq B_{\frac{\epsilon}{4}}(E_k(b,K))$ for every $b \in \mathrm{A}(\Gamma,X,\mu)$ where $B_r(A) = \{ \nu \in M_s(K^\Gamma): \delta_K (\nu,\rho) < r$ for some $\rho \in A\}$. Notice that $E_k(b,K) = \bigcup_{\substack{L \subseteq K,\\ |L| = k}} E(b,L)$. Fix a set $L$ of size $k$ and choose $N$ large enough such that if $n\geq N$ then $d_{\mathcal{K}_L}\left(\overline{E(a_n,L)},\overline{E(a,L)} \right) < \frac{\epsilon}{4}$ where $d_{\mathcal{K}_L}$ is the Hausdorff distance in $\mathcal{K}(M_s(L^\Gamma))$. Since the construction is independent of the set chosen to realize $L$, we have in fact $d_{\mathcal{K}_L}\left(\overline{E(a_n,L)},\overline{E(a,L)} \right) < \frac{\epsilon}{4}$ for every finite set $L$ of size $k$. For a fixed finite $L \subseteq K$ let $E_L(b,K) =  \{ (\Phi^{b,\phi})_* \mu: \phi: X \to K \mbox{ measurable, } \phi(X) \subseteq L \} $. Then for any $b,c \in \mathrm{A}(\Gamma,X,\mu)$ we have \[d_\mathcal{K}\left(\overline{E_L(b,K)}, \overline{E_L(c,K)} \right) = d_{\mathcal{K}_L} \left( \overline{E(b,L)}, \overline{E(c,L)} \right), \] so that when $n\geq N$, \begin{align*} d_{\mathcal{K}}\left(\overline{E_k(a_n,K)}, \overline{E_k(a,K)} \right) &= d_{\mathcal{K}} \left( \bigcup_{\substack{L \subseteq K\\ |L| = k}} \overline{E(a_n,L)},\bigcup_{\substack{L \subseteq K \\ |L| = k}} \overline{E(a,L)} \right) \\ &\leq \sup_{\substack{L \subseteq K \\ |L| = k}} d_{\mathcal{K}_L} \left( \overline{E(a_n,L)}, \overline{E(a,L)} \right) < \frac{\epsilon}{4}. \end{align*} Therefore when $n \geq N$, \begin{align*} d_{\mathcal{K}}\left(\overline{E(a_n,K)},\overline{E(a,K)}\right)  &\leq d_{\mathcal{K}}\left(\overline{E(a_n,K)}, \overline{E_k(a_n,K)}\right) + d_{\mathcal{K}}\left( \overline{E_k(a_n,K)}, \overline{E_k(a,K)}\right) \\ & \hspace{1 in}+ d_{\mathcal{K}}\left(\overline{E_k(a,K)}, \overline{E(a,K)}\right) \\ &< \frac{3 \epsilon}{4}. \end{align*}

 \end{proof}

\begin{lemma} \label{lem7} Let $L$ be a finite set of size $k$. Then for each finite set $(A_p)_{p=1}^q$ of basic clopen sets $A_p \subseteq L^\Gamma$ and $\epsilon > 0$ there is $\delta > 0$ such that if $d(a,b) < \delta$ then for all $\phi \in L(X,\mu,L)$ there exists $\psi \in L(X,\mu,L)$ such that $|(\Phi_L^{a,\phi})_* \mu(A_p) - (\Phi_L^{b,\psi})_* \mu(A_p) | < \epsilon$ for all $p \leq q$. \end{lemma}

\begin{proof}

Write $A_p = \bigcap_{\gamma \in F_p} \pi_\gamma^{-1}(\jmath_p(\gamma))$ for some $F_p \subseteq \Gamma$ finite, $\jmath: F_p \to k$ and fix $\epsilon > 0$. Choose a finite $F \subseteq \Gamma$ with $(F_p)^2 \subseteq F$ for all $p \leq q$. We may assume the identity $e \in F$. Suppose $d(a,b) < \frac{\delta}{2^{|F| + k^{|F|}}}$; we will specify a value for $\delta$ later. Now fix $\phi:X \to k$ and let $B_i = \phi^{-1}(i)$. Given $\eta: F \to k$, let $B_\eta = \bigcap_{\gamma \in F} \gamma^{a} B_{\eta(\gamma)}$. We can then find a partition $\{D_\eta\}_{\eta \in k^F}$ such that \[ | \mu( \gamma^{a}B_{\eta_1} \cap B_{\eta_2}) - \mu( \gamma^b D_{\eta_1} \cap D_{\eta_2})| < \delta \] for all $\eta_1,\eta_2 \in k^F$ and $\gamma \in F$. Define $\psi: X \to k$, by $\psi(y) = l$ if $y \in D_{\eta}$ for some $\eta$ with $\eta(e) = l$. Furthermore, for each $l \leq k$ let $D_l = \bigsqcup \{D_\eta: \eta \in k^F$ and $\eta(e) = l \} = \psi^{-1}(l)$. For each $J \subseteq F$ and $\sigma \in k^J$ let $D_\sigma = \bigsqcup \{ D_\eta: \eta \in k^F$ and $\sigma \sqsubseteq \eta \}$, where $\sigma \sqsubseteq \eta$ means $\eta$ extends $\sigma$ and let $\tilde{D}_\sigma = \bigcap_{\gamma \in J} \gamma^b D_{\sigma(\gamma)}$. Furthermore if $\gamma \in \Gamma$, $J \subseteq \Gamma$ and $\sigma \in k^J$ let $\gamma \cdot \sigma \in k^{\gamma J}$ be given by $(\gamma \cdot \sigma)(\delta) = \sigma( \gamma^{-1} \delta)$. For $\sigma \in K^{F_p}$ and $\gamma \in F_p$ we have \begin{align*} | \mu(\gamma^b D_{\sigma} \cap D_{\gamma \cdot \sigma}) - \mu(\gamma^a B_{\sigma} \cap B_{\gamma \cdot \sigma})| & \leq \sum_{(\eta \in k^{F}: \sigma \sqsubseteq \eta)} \sum_{(\eta' \in k^{F}: \gamma \cdot \sigma \sqsubseteq \eta')} |\mu( \gamma^b D_\eta \cap D_{\eta'}) - \mu( \gamma^a B_{\eta} \cap B_{\eta'})| \\  & \leq \delta (k^{|F|})^2 \end{align*}

In particular, setting $\gamma = e$ we see  $|\mu(B_\sigma) - \mu(D_\sigma)| < \delta k^{2 |F|}$ for every $\sigma: F_p \to k$. Since $\gamma^a B_\sigma = B_{\gamma \cdot \sigma} = \gamma^a B_\sigma \cap B_{\gamma \cdot \sigma}$ we have \begin{align*} |\mu(D_\sigma) - \mu( \gamma^b D_\sigma \cap D_{\gamma \cdot \sigma})| & \leq | \mu(D_\sigma) - \mu(\gamma^a B_\sigma)| + |\mu(\gamma^a B_\sigma \cap B_{\gamma \cdot \sigma}) - \mu(\gamma^b D_\sigma \cap D_{\gamma \cdot \sigma})| \\ & = | \mu(D_\sigma) - \mu(B_\sigma)| + |\mu(\gamma^a B_\sigma \cap B_{\gamma \cdot \sigma}) - \mu(\gamma^b D_\sigma \cap D_{\gamma \cdot \sigma})|   \\ &< 2 \delta k^{2 |F|} \end{align*} and also \begin{align*} |\mu(D_{\gamma \cdot \sigma}) - \mu( \gamma^b D_\sigma \cap D_{\gamma \cdot \sigma})| & \leq | \mu(D_{\gamma \cdot \sigma}) - \mu(B_{\gamma \cdot \sigma})| + |\mu(\gamma^a B_\sigma \cap B_{\gamma \cdot \sigma}) - \mu(\gamma^b D_\sigma \cap D_{\gamma \cdot \sigma})| \\ &< 2 \delta k^{2 |F|} \end{align*} Therefore \begin{align} \mu( (\gamma^b D_\sigma) \triangle (D_{\gamma \cdot \sigma})) & = \mu( \gamma^b D_\sigma) + \mu(D_{\gamma \cdot \sigma}) - 2 \mu( \gamma^b D_\sigma \cap D_{\gamma \cdot \sigma})  \nonumber \\ & \leq |\mu(D_{\gamma \cdot \sigma}) - \mu( \gamma^b D_\sigma \cap D_{\gamma \cdot \sigma})| +  |\mu(D_{\gamma \cdot \sigma}) - \mu( \gamma^b D_\sigma \cap D_{\gamma \cdot \sigma})| \nonumber \\ & < 4 \delta k^{2 |F|}   \end{align}

Since $(D_\eta)_{\eta \in k^F}$ is a partition of $X$ and $(F_p)^2 \subseteq F$ we have

\[ D_{\jmath_p} = \bigsqcup_{\substack{ \eta \in k^F \\ \jmath_p \sqsubseteq \eta }} D_\eta = \bigcap_{\gamma \in F_p} \bigsqcup_{\substack{\sigma \in k^{\gamma F_p}\\ \sigma(\gamma) = \jmath_p(\gamma)}} D_\sigma = \bigcap_{\gamma \in F_p} \bigsqcup_{\substack{\sigma \in k^{F_p} \\ \sigma(e) = \jmath_p(\gamma)}} D_{\gamma \cdot \sigma}.\]

Now, by $(2)$, \begin{equation} \mu \left( \left( \bigcap_{\gamma \in F_p} \bigsqcup_{\substack{\sigma \in k^{F_p} \\ \sigma(e) = \jmath(\gamma)}} D_{\gamma \cdot \sigma} \right) \triangle \left( \bigcap_{\gamma \in F_p} \bigsqcup_{\substack{\sigma \in k^{F_p} \\ \sigma(e) = \jmath_p(\gamma)}} \gamma^b D_\sigma \right) \right) < (|F_p| k^{|F_p|}) ( 4 \delta k^{2 |F|}). \end{equation}
Note that $ \bigcap_{\gamma \in F_p} \bigsqcup_{\substack{\sigma \in k^{F_p} \\ \sigma(e) = \jmath_p(\gamma)}} \gamma^b D_\sigma = \bigcap_{\gamma \in F_p} \gamma^b D_{\jmath_p(\gamma)} = \tilde{D}_{\jmath_p}$, so $(3)$ reads $|\mu(D_{\jmath_p}) - \mu(\tilde{D}_{\jmath_p})| < (|F_p| k^{|F_p|}) ( 4 \delta k^{2 |F|})$.

Moreover, \begin{align*} (\Phi^{b,\psi}_L)_* \mu(A_p) & = \mu(\{x: \Phi^{b,\psi}_L(x) \in A_p \}) \\ & = \mu(\{x: \Phi^{b,\psi}_L(x)(\gamma) = \jmath_p(\gamma) \mbox{ for all } \gamma \in F_p \}) \\ & =  \mu(\{x: \psi((\gamma^{-1})^b x) = \jmath_p(\gamma) \mbox{ for all } \gamma \in F_p \}) \\ & =  \mu(\{x:  x \in \gamma^b \psi^{-1}( \jmath_p(\gamma)) \mbox{ for all } \gamma \in F_p \}) \\ & = \mu \left( \bigcap_{\gamma \in F_p} \gamma^b D_{\jmath_p(\gamma)} \right) \\ & = \mu(\tilde{D}_{\jmath_p}). \end{align*}

Similarly, $(\Phi^{a,\phi}_L)_*\mu(A_p) = \mu(B_{\jmath_p})$. So we finally have \begin{align*} | (\Phi^{b,\psi}_L)_* \mu(A_p) - (\Phi^{a,\phi}_L)_*\mu(A_p) |  & = |\mu(\tilde{D}_{\jmath_p}) - \mu(B_{\jmath_p})| \\ & \leq |\mu(\tilde{D}_{\jmath_p}) - \mu(D_{\jmath_p})| + |\mu(D_{\jmath_p}) - \mu(B_{\jmath_p})| \\ &< (|F_p| k^{|F_p|}) ( 4 \delta k^{2 |F|}) +  2 \delta k^{2 |F|}.\end{align*}

Since $k$ is fixed in advance, $|F_p| \leq |F|$ and $F$ depends only on $(A_p)_{p=1}^q$, it is clear that $\delta$ can be chosen so $(|F_p| k^{|F_p|}) ( 4 \delta k^{2 |F|}) +  2 \delta k^{2 |F|} < \epsilon$ for all $p \leq q$. \end{proof}

We can now prove the main result of this section. 

\begin{theorem} $\tau_1 = \tau_2$. \end{theorem}

\begin{proof} Suppose that $a_n \to a$ in $\tau_1$. We need to prove $\Phi(a_n) \to \Phi(a)$ in $\mathcal{K}(M_s(K^\Gamma))$. By Lemma \ref{lem6} it suffices to fix a finite set $L$ and show $\overline{E(a_n,L)} \to \overline{E(a,L)}$ in $\mathcal{K}(M_s(L^\Gamma))$. Let $k = |L|$. Write $E_n = E(a_n,L)$ and $E = E(a,L)$. As before, if we let $\mathcal{A}_L = (A^L_i)_{i=1}^\infty$ be the collection of clopen subsets of $L^\Gamma$ of the form $\bigcap_{\gamma \in F} \pi_\gamma^{-1}(j_\gamma)$ for a finite $F \subseteq \Gamma$ and $j_\gamma \leq k$, then \[ \delta_L(\nu,\rho) = \sum_{i=1}^\infty \frac{1}{2^i} |\nu(A^L_i) - \rho(A^L_i)| \] is a compatible metric on $M_s(L^\Gamma)$. Fix $\epsilon > 0$ in order to show that eventually $d_L(\overline{E_n},\overline{E}) < \epsilon$, where $d_L$ is the Hausdorff distance in $\mathcal{K}(M_s(L^\Gamma))$ constructed from $\delta_L$. Choose $N$ sufficiently large that $\sum_{i=N}^\infty \frac{1}{2^i} < \frac{\epsilon}{2}$. By Lemma \ref{lem7} there is $\delta > 0$ such that if $d(a,b) < \delta$ then for each $i \leq N$ and all $\phi \in L(X,\mu,L)$ there exists $\psi \in L(X,\mu,L)$ such that $|(\Phi_L^{a,\phi})_* \mu(A^L_i) - (\Phi_L^{b,\psi})_* \mu(A^L_i) | < \frac{\epsilon}{2}$. Thus if $M$ is large enough that $d(a_n,a) < \delta$ for $n \geq M$, we have $d_L( \overline{E_n}, \overline{E}) < \epsilon$.\\
\\
Now suppose $\Phi(a_n) \to \Phi(a)$ in $\mathcal{K}(M_s(K^\Gamma))$. Fix $r,q$ and $\epsilon > 0$ in order to show that eventually $d_H(C_{r,q}(a_n),C_{r,q}(a)) < \epsilon$. Choose $q$ distinct points $(x_p)_{p=1}^q \in K$ and let $(D_p)_{p=1}^q$ be a family of disjoint clopen subsets of $K$ with $x_p \in D_p$. Now let $M$ be large enough that all sets of the form $\pi_{\gamma_s}^{-1}(D_p) \cap \pi_e^{-1}(D_t)$ for $s \leq r$ and $p,t \leq q$ appear as some $A^K_i$ for $i \leq M$ in our previously chosen clopen basis $\mathcal{A}_K$. Then choose $N$ large enough that when $n \geq N$, $d_{\mathcal{K}}( \Phi(a_n), \Phi(a) ) < \frac{\epsilon}{2^{M}}$. Then for each $\phi \in L(X,\mu,K)$ we have $\psi \in L(X,\mu,K)$ such that $\delta_K((\Phi^{a_n,\phi})_* \mu, (\Phi^{a,\psi})_* \mu) <  \frac{\epsilon}{2^{M}}$. So in particular,  if $n \geq N$ then for each $\phi \in L(X,\mu,K)$ there exists $\psi \in L(X,\mu,K)$ such that \[ | (\Phi^{a_n,\phi})_* \mu(\pi_{\gamma_s}^{-1}(D_p) \cap \pi_e^{-1}(D_t)) - (\Phi^{a,\psi})_* \mu(\pi_{\gamma_s}^{-1}(D_p) \cap \pi_e^{-1}(D_t))| < \epsilon \] for all $p,t \leq q$ and $s \leq r$.\\
\\
Now suppose $n \geq N$ and let $(B_p)_{p=1}^q$ be a partition of $X$. Define $\phi: X \to K$ by taking $\phi(x) = x_p$ for the unique $p \leq q$ with $x \in B_p$ so by the previous paragraph we have a corresponding $\psi$. Observe that for all $\gamma \in \Gamma$ we have \begin{align*} \mu( \gamma^{a_n} B_p \cap B_t) &= \mu( \gamma^{a_n} \phi^{-1}(D_p) \cap \phi^{-1}(D_t) ) \\ & = \mu(\{x: \phi((\gamma^{a_n})^{-1} x) \in D_p \mbox{ and } \phi(x) \in D_t \}) \\ & = \mu( \{ x: \Phi^{\phi,a_n}(x)(\gamma) \in D_p \mbox{ and } \Phi^{\phi,a_n}(x)(e) \in D_t \}) \\ & = \mu( \{ x: \Phi^{\phi,a_n}(x) \in \pi_{\gamma}^{-1}(D_p) \mbox{ and } \Phi^{\phi,a_n}(x) \in \pi_e^{-1}(D_t) \}) \\ & = \mu( \{x : \Phi^{\phi,a_n}(x) \in \pi_{\gamma}^{-1}(D_p) \cap \pi_1^{-1}(D_t) \} ) \\ & = (\Phi^{\phi,a_n})_* \mu( \pi_{\gamma}^{-1}(D_p) \cap \pi_1^{-1}(D_t) ). \end{align*} Similarly letting $H_p = \psi^{-1}(D_p)$ we have $\mu( \gamma^a H_p \cap H_t) = (\Phi^{\psi,a_n})_* \mu( \pi_{\gamma}^{-1}(D_p) \cap \pi_1^{-1}(D_t) )$. Thus for all $p,t \leq q$ and $s \leq r$,

\[ |\mu( \gamma_s^{a_n} B_p \cap B_t ) - \mu(\gamma_s^a H_p \cap H_t)| = |(\Phi^{\phi,a_n})_* \mu( \pi_{\gamma+s}^{-1}(D_p) \cap \pi_e^{-1}(D_t) ) - (\Phi^{\psi,a_n})_* \mu( \pi_{\gamma_s}^{-1}(D_p) \cap \pi_e^{-1}(D_t) )| < \epsilon.\] We have shown that when $n \geq N$, $C_{r.q}(a_n) \subseteq B_{\epsilon}(C_{r,q}(a))$. The argument that eventually $C_{r,q}(a) \subseteq B_\epsilon( C_{r,q}(a_n))$ is identical. \end{proof}

\subsection{Topology on the space of stable weak equivalence classes.}

Let $\mathrm{A}_{\sim_s}(\Gamma,X,\mu)$ be the space of stable weak equivalence classes and let $\iota$ be the trivial action of $\Gamma$ on an standard probability space. By Lemma 3.7 in \cite{RTD}, we have $a \prec_s b$ if and only if $a \prec \iota \times b$. Moreover, Theorem 1.1 in \cite{RTD} says that $\overline{ E (a \times \iota,K)} = \cch(E(a,K))$, where $M_s(K^\Gamma)$ carries its natural topological convex structure as a compact convex subset of a Banach space.  Letting $\Psi:\mathrm{A}(\Gamma,X,\mu) \to \mathcal{K}(M_s(K^\Gamma))$ be the map $a \mapsto \cch(E(a,K)) $ we have $\Psi(a) = \Psi(b)$ if and only if $a \sim_s b$. Tucker-Drob gives $\mathrm{A}_{\sim_s}(\Gamma,X,\mu)$ the initial topology induced by $\Psi$, in which it is a compact Polish space. Thus we have $a_n \to a$ in the topology of $\mathrm{A}_{\sim_s}(\Gamma,X,\mu)$ if and only if $a_n \times \iota \to a \times \iota$ in the topology of $\mathrm{A}_\sim(\Gamma,X,\mu)$. Therefore we can introduce a metric $d_s$ on $\mathrm{A}_{\sim_s}(\Gamma,X,\mu)$ by setting $d_s(a,b) = d(a \times \iota,b \times \iota)$.

\section{$\mathrm{A}_\sim(\Gamma,X,\mu)$ as a weak convex space.}

We now describe how to give $\mathrm{A}_\sim(\Gamma,X,\mu)$ the structure of a weak convex space. Given $t \in [0,1]$ and $a,b \in \mathrm{A}_\sim(\Gamma,X,\mu)$ we let $c \in A\left(\Gamma,X_1 \sqcup X_2,t \mu_1 + (1-t) \mu_2 \right)$ be the disjoint sum of representative actions $a$ and $b$ on the disjoint union of two copies $X_1$ and $X_2$ of $X$ with the first copy carrying a copy of the measure $\mu$ weighted by $t$ and the second copy carrying a copy of $\mu$ weighted by $(1-t)$. To get an action in $\mathrm{A}(\Gamma,X,\mu)$ we need to choose an isomorphism of $(X,\mu)$ with $\left(X_1 \sqcup X_2, t \mu_1 + (1-t) \mu_2 \right)$, but the weak equivalence class of $c$ does not depend on this or on the representatives we chose. So we have a well-defined binary operation $\mathrm{A}_\sim(\Gamma,X,\mu)^2 \to \mathrm{A}_\sim(\Gamma,X,\mu)$. Call this $cc_t$. It is clear that $\mathrm{(1)}$, $\mathrm{(3)}$ and $\mathrm{(4)}$ of Definition \ref{def1} are satisfied, so $\mathrm{A}_\sim(\Gamma,X,\mu)$ is a weak convex space. Moreover, we have the following.

\begin{proposition} $\mathrm{A}_\sim(\Gamma,X,\mu)$ is a topological weak convex space. \end{proposition}

\begin{proof} We must show that $cc$ is continuous. Suppose that $t_j \to t$ in $[0,1]$ and $a_j \to a$ and $b_j \to b$ in the topology of $A_\sim(\Gamma,Y,\mu)$. Write $c_j =t_j a_j + (1-t_j) b_j $ and $c = t a + (1-t)b$.  Fixing $l,m \in \mathbb{N}$ write $C(d)$ for $C_{l,m}(d)$. We need to prove that for every $\epsilon >0$ there is $J$ so that if $j > J$ then we have $d_H(C(c_j),C(c)) < \epsilon$, where $d_H$ is the Hausdorff distance in $[0,1]^{l \times m^2}$.\\
\\
First we must show that for sufficiently large $j$, for every partition $B_1,\ldots,B_l$ of $Y$ there is a partition $D_1,\ldots,D_l$ of $Y$ depending on $j$ such that for all $s,t \leq l$ and $p \leq m$, \[ |\mu(\gamma_p^{c_j} D_s \cap D_t) - \mu(\gamma_p^{c} B_s \cap B_t)| < \epsilon. \]

Choose $J_1$ so that if $j > J_1$ then $ |t_j - t| < \frac{\epsilon}{6}$. Choose $J_2 > J_1$ so if $j > J_2$ then $d_H(C_{a_j},C_a) < \frac{\epsilon}{6}$ and $d_H(C_{b_j},C_b) < \frac{\epsilon}{6}$. Fix $j > J_2$. Writing $\theta$ for the isomorphism from $(Y_1 \sqcup Y_2, t \mu + (1-t) \mu)$ to $(Y,\mu)$ and $\theta_j$ for the isomorphism from $(Y_1 \sqcup Y_2, t_j \mu + (1-t_j) \mu)$ to $(Y,\mu)$ we have a partition $(B_{s,i})_{s=1}^l$ of $Y_i$ given by $B_{s,i} = \theta^{-1}(B_s) \cap Y_i$. So we can find a partition $(D_{s,i})_{s=1}^l$ of $Y_i$ such that for all $p \leq m$ and all $s,t \leq l$ we have \[ |\mu(\gamma_p^{a_j} D_{s,1} \cap D_{t,1}) - \mu(\gamma_p^{a} B_{s,1} \cap B_{t,1}) | < \frac{\epsilon}{6}\] and \[ |\mu(\gamma_p^{b_j} D_{s,2} \cap D_{t,2}) - \mu(\gamma_p^{b} B_{s,2} \cap B_{t,2}) | < \frac{\epsilon}{6}\] 

Now, let $D_s = \theta_j(D_{s,1} \sqcup D_{s,2})$. Note that since each $\theta_j(Y_i)$ is $c_j$ invariant, \begin{align*} \mu(\gamma_p^{c_j} D_s \cap D_t) &= \mu( \gamma_p^{c_j} \theta_j(D_{s,1}) \cap \theta_j(D_{t,1})) + \mu( \gamma_p^{c_j} \theta_j(D_{s,2}) \cap \theta_j(D_{t,2})) \\ & =  \mu( \theta_j( \gamma_p^{a_j} D_{s,1} \cap D_{t,1} ) ) + \mu( \theta_j( \gamma_p^{b_j} D_{s,2} \cap D_{t,2} ) ) \\ & =  t_j \mu(\gamma_p^{a_j} D_{s,1} \cap D_{t,1} ) + (1-t_j) \mu(\gamma_p^{b_j} D_{s,2} \cap D_{t,2} ).   \end{align*}

Similarly since $\theta(Y_i)$ is $c$-invariant we have \begin{align*} \mu(\gamma_p^{c} B_s \cap B_t)&= \mu(     \gamma_p^{c} \theta(B_{s,1}) \cap \theta(B_{t,1})) + \mu(     \gamma_p^{c} \theta(B_{s,2}) \cap \theta(B_{t,2})) \\ & =  \mu( \theta( \gamma_p^{a} B_{s,1} \cap B_{t,1} ) ) + \mu( \theta( \gamma_p^{b} B_{s,2} \cap B_{t,2} ) ) \\ & = t \mu(\gamma^{a}_p \cap B_{s,1} \cap B_{t,1}) + (1-t) \mu(\gamma^{b}_p \cap B_{s,2} \cap B_{t,2}) \end{align*}

Note that if $|x_1 - x_2| < \delta$ and $|y_1 - y_2| < \delta$ then $|x_1y_1 - x_2 y_2| < 3 \delta$. So our assumptions guarantee that we have \[|t_j \mu(\gamma_p^{a_j} D_{s,1} \cap D_{t,1} ) -  t \mu(\gamma^{a}_p \cap B_{s,1} \cap B_{t,1})| < \frac{\epsilon}{2} \] and \[|(1-t_j )\mu(\gamma_p^{b_j} D_{s,2} \cap D_{t,2} ) -  (1-t) \mu(\gamma^{b}_p \cap B_{s,2} \cap B_{t,2})| < \frac{\epsilon}{2} \] hence \[ |\mu(\gamma_p^{c_j} D_s \cap D_t) - \mu(\gamma_p^{c} B_s \cap B_t)| < \epsilon \] as claimed.\\
\\
Now we must show that for sufficiently large $j$, every partition $B_1,\ldots,B_l$ of $Y$ there is a partition $D_1,\ldots,D_l$ of $Y$ depending on $j$ such that for all $s,t \leq l$ and $p \leq m$ we have \[|\mu(\gamma_p^{c} D_s \cap D_t) - \mu(\gamma_p^{c_j} B_s \cap B_t)| < \epsilon.\] The argument is similar to the previous step, so we omit it. \end{proof}

\begin{corollary} $A_\sim(\Gamma,Y,\mu)$ is path connected. \end{corollary}

\begin{corollary} $A_\sim(\Gamma,Y,\mu)$ is uncountable. \end{corollary}

We now record a lemma which will be useful later, guaranteeing that the metric on $\mathrm{A}_\sim(\Gamma,X,\mu)$ behaves nicely with respect to the convex structure.

\begin{lemma}\label{lem3} For any convex set $K \subseteq \mathrm{A}_\sim(\Gamma,X,\mu)$ the function $d( \cdot,K) = \inf_{b \in K} d( \cdot,b)$ is convex. \end{lemma}

\begin{proof} Let $x,y \in \mathrm{A}_\sim(\Gamma,X,\mu)$ and consider $tx + (1-t)y$. Fix $n,k$ and write $C(a)$ for $C_{n,k}(a)$. It suffices to show that \[ \inf_{b \in K} d_H( C(tx + (1-t)y), C(b)) \leq t( \inf_{b \in K} d_H( C(x), C(b) ) + (1-t)( \inf_{b \in K} d_H( C(y),C(b)))\] where $d_H$ is the Hausdorff distance in the space $[0,1]^{n \times k^2}$. Fix $\epsilon > 0$. It suffices to find $a \in K$ with 
\begin{equation} d_H( C(tx + (1-t)y), C(a)) \leq t( \inf_{b \in K} d_H( C(x), C(b) ) + \epsilon ) + (1-t)( \inf_{b \in K} d_H( C(y),C(b)) + \epsilon). \end{equation}

Choose $c \in K$ with $d_H(C(x),C(c)) < \inf_{b \in K} d_H( C(x), C(b) ) + \epsilon$ and choose $d \in K$ with $d_H(C(x),C(d)) < \inf_{b \in K} d_H( C(y), C(b) ) + \epsilon$. Note that since $K$ is convex, $tc + (1-t)d \in K$. We claim

\[ d_H( C(tx + (1-t) y), C(tc + (1-t) d)) \leq t d_H(C(x),C(c)) + (1-t) d_H(C(y),C(d)) \],

which implies $(4)$. Let $\delta > 0$, it then suffices to show 

\begin{equation} d_H( C(tx + (1-t) y), C(tc + (1-t) d)) \leq t( d_H(C(x),C(c)) + \delta) + (1-t) (d_H(C(y),C(d)) + \delta). \end{equation}

Let $X_1$ and $X_2$ be two copies of $X$ and $\nu$ be the measure on $X_1 \sqcup X_2$ given by $t (\mu \upharpoonright X_1) + (1-t) (\mu \upharpoonright X_2)$. Let $\mathcal{P} = (P_i)_{i=1}^k$ be a partition of $X_1 \sqcup X_2$. This induces a partition $\mathcal{P}_1 = (P^1_i)_{i=1}^k$ of $X_1$ given by $P^1_i = P_i \cap X_1$ and similarly we have a partition $\mathcal{P}_2 = (P^2_i)_{i=1}^k$ of $X_2$. We can find a partition $\mathcal{Q}_1 = (Q^1_i)_{i=1}^k$ of $X_1$ such that for $m \leq n$ and $i,j \leq k$ we have

\[ | \mu(\gamma_m^x P^1_i \cap P^1_j) - \mu(\gamma_m^c Q^1_i \cap Q^1_j)| < d_H(C(x),C(c)) + \delta \]

and similarly we can find a partition $\mathcal{Q}_2 = (Q^2_i)_{i=1}^k$ of $X_2$ such that for $m \leq n$ and $i,j \leq k$ we have

\[ | \mu(\gamma_m^y P^2_i \cap P^2_j) - \mu(\gamma_m^d Q^2_i \cap Q^2_j)| < d_H(C(y),C(d)) + \delta. \]

Let $\mathcal{Q} = (Q_i)_{i=1}^k$ be the partition of $X_1 \sqcup X_2$ given by $Q_i = Q^1_i \sqcup Q^2_i$. Write $t( d_H(C(x),C(c)) + \delta) + (1-t) (d_H(C(y),C(d)) + \delta) = r$. Then for all $m \leq n$ and $i,j \leq k$ we have

\begin{align*} |\nu( \gamma_m^{tx + (1-t)y} P_i \cap P_j) - \nu( \gamma_m^{tc + (1-t)d} Q_i \cap Q_j)| & \leq  |t\mu(\gamma_m^x P^1_i \cap P^1_j) - t\mu(\gamma_m^c Q^1_i \cap Q^1_j)| \\ & +  | (1-t)\mu(\gamma_m^y P^2_i \cap P^2_j) - (1-t)\mu(\gamma_m^d Q^2_i \cap Q^2_j)| \\ &\leq r \end{align*}

We have shown that $C(tx + (1-t) y) \subseteq B_r(C(tc + (1-t)d))$. The argument that $C(tc + (1-t)d) \subseteq B_r(C(tx + (1-t)y))$ is identical, so we omit it. Thus we conclude $d_H( C(tx + (1-t) y), C(tc + (1-t) d)) \leq r$ and $(5)$ holds. \end{proof}

We note that $\mathrm{A}_\sim(\Gamma,X,\mu)$ in fact has additional structure in that it admits convex combinations of infinitely many elements. We first consider the case of a countable convex combination. If $\lambda_i \in [0,1]$ are such that $\sum_{i=1}^\infty \lambda_i = 1$ and $a_i \in \mathrm{A}_\sim(\Gamma,X,\mu)$ then we can naturally define an action $\sum_{i=1}^\infty \lambda_i a_i$ on the disjoint sum $\bigsqcup_{i=1}^\infty X_i$ with the $i$ copy of $X$ weighted by $\lambda_i$. It remains to check that this is independent of the choice of representatives $a_i$.

\begin{proposition} \label{prop2} If $a_i \prec b_i$ for all $i$, then $\sum_{i=1}^\infty \lambda_i a_i \prec \sum_{i=1}^\infty \lambda_i b_i$.\end{proposition}

\begin{proof} Let $A_1,\ldots,A_k \subseteq \bigsqcup_{m=1}^\infty X_m$, $\epsilon > 0$ and $F \subseteq \Gamma$ finite be given. Choose $N$ such that $\sum_{m=N}^\infty \lambda_m < \frac{\epsilon}{2}$. For each $m < N$, consider the partition $A^m_1,\ldots,A^m_k$ of $X_m$ given by $A^m_i = A_i \cap X_m$. We can find for each $m<N$ a partition $B^m_1,\ldots,B^m_k$ such that for all $\gamma \in F$ and $i,j \leq k$ we have \[ |\mu(\gamma^{a_i} A^m_i \cap A^m_j) - \mu(\gamma^{b_i} B^m_i \cap B^m_j)| < \frac{\epsilon}{2}.\] Let $B_i = \bigsqcup_{m=1}^\infty B^m_i$. Then

\begin{align*} \left \vert \mu\left(\gamma^{\sum_{m=1}^\infty \lambda_m a_m} A_i \cap A_j \right) - \mu\left(\gamma^{\sum_{m=1}^\infty \lambda_m b_m} B_i \cap B_j \right) \right \vert & \leq |\sum_{m=1}^N \lambda_m \mu(\gamma^{a_m} A^m_i \cap A^m_j) - \sum_{m=1}^N \lambda_m \mu(\gamma^{b_m} B^m_i \cap B^m_j)| \\ \hspace{1 in} &+  |\sum_{m=M}^\infty \lambda_m \mu(\gamma^{a_m} A^m_i \cap A^m_j) - \sum_{m=M}^\infty \lambda_m \mu(\gamma^{b_m} B^m_i \cap B^m_j)| \\ & \leq \sum_{m=1}^N \lambda_m |\mu(\gamma^{a_i} A^m_i \cap A^m_j) - \mu(\gamma^{b_i} B^m_i \cap B^m_j)| + \frac{\epsilon}{2} \\ & \leq \frac{\epsilon}{2} \left( \sum_{m=1}^N \lambda_m \right) + \frac{\epsilon}{2} \leq \epsilon. \end{align*} \end{proof}

It is in fact possible to define integrals of weak equivalence classes of actions over a probability measure. Let $(Z,\eta)$ be a probability space and suppose that for each $z$ we have a probability space $(Y_z,\nu_z)$ and a measure-preserving action $\Gamma \curvearrowright^{a_z} (Y_z,\nu_z)$ such that the map $z \mapsto [a_z]$ from $(Z,\eta)$ to $A^*_\sim(\Gamma)$ is measurable, where $[a_z]$ is the weak equivalence class of $a_z$. Note that we do not require $(X_z,\nu_z)$ or $(Z,\eta)$ to be standard. Let $Y = \bigsqcup_{z \in Z} Y_z$ and put a measure $\nu$ on $Y$ by taking $\nu(A) = \int_Z \nu_z(A \cap Y_Z) d \eta(z)$. $Y$ will be a standard probability space isomorphic to $(X,\mu)$ if $(Z,\eta)$ is standard or $\eta$-almost all $(Y_z,\nu_z)$ are standard. Let $\Gamma \curvearrowright^a (Y,\nu)$ be given by letting $\Gamma$ act like $a_z$ on $Y_z$. We write $a = \int_Z a_z d \eta(z)$. We then have a map $\phi: Y \to Z$ given by letting $\phi(y)$ be the unique $z$ such that $y \in Y_z$. This is clearly a factor map from $a$ to $\iota_{Z,\eta}$ and $\nu = \int_Z \nu_z d \eta(z)$ is the disintegration of $\nu$ over $\eta$ via $\phi$. Thus Theorem 3.12 in \cite{RTD} guarantees that if $b_z$ are actions of $\Gamma$ on $(Y_z,\nu_z)$ with $b_z \sim a_z$ then if $b = \int_Z b_z d \eta(z)$ we have $a \sim b$. Therefore this construction gives a well-defined weak equivalence class of actions of $\gamma$. If we restrict $(Y_z,\nu_z)$ to be standard, then we in fact have a mapping from the space $M(\mathrm{A}_\sim(\Gamma,X,\mu))$ of probability measures on $\mathrm{A}_\sim(\Gamma,X,\mu)$ to $\mathrm{A}_\sim(\Gamma,X,\mu)$.

\begin{lemma}\label{lem11} For any $n,k$, and $(Z,\eta)$ and measurable assignment $z \mapsto a_z$, we have $C_{n,k}\left( \int_Z a_z d \eta(z) \right) \subseteq \cch \left( \bigcup_{z \in Z} C_{n,k}(a_z) \right)$. \end{lemma}

\begin{proof} Fix $n,k$ and let $a = \int_Z a_z d \eta(z)$. Let $(X_z,\mu_z)$ by the underlying measure space of $a_z$. Let $\mathcal{L}$ be a countable dense subset of $\mathrm{MALG} \left ( \bigsqcup_{z \in Z} X_z, \int_Z \mu_z d \eta(z) \right )$, so that $\mathcal{L}^k$ is dense in the space of $k$-partitions of $\bigsqcup_{z \in Z} X_z$. Then $\{M_{\mathcal{A}}(a) \}_{\mathcal{A} \in \mathcal{L}^n }$ is dense in $C_{n,k} \left(\int_Z a_z d \eta(z) \right)$, so it suffices to show that each $M_{\mathcal{A}}(a) \in  \cch \left( \bigcup_{z \in Z} C(a_z) \right)$. For each $\mathcal{A}$, the function $f_\mathcal{A}:Z \to \mathbb{R}^{n \times k \times k}$ given by $z \mapsto M_{\mathcal{A}_z}(a_z)$ is a Borel function, where $\mathcal{A}_z$ is the partition of $X_z$ given by $(A \cap X_z)_{A \in \mathcal{A}}$. Thus $M_{\mathcal{A}}(a) = \int_Z f_{\mathcal{A}}(a) d \eta(z)$. We may assume that $Z$ carries a Polish topology such that $f_{\mathcal{A}}$ is continuous for all $\mathcal{A} \in \mathcal{L}^n$. Choose a sequence of measures $(\nu_i)_{i=1}^\infty$ such that $\nu_i$ has finite support and $\nu_i \to \eta$ in the topology of $M(Z)$, the space of all Borel probability measures on $Z$. If we write $\nu_i = \sum_{j=1}^{j(i)} \alpha_j \delta_{z_j}$ then \[ \int_Z f_{\mathcal{A}}(z) d \nu_i(z) = \sum_{j=1}^{j(i)} \alpha_j f_{\mathcal{A}}(z_j) \in \ch \left( \bigcup_{z \in A} C(a_z) \right). \] Since $\nu_i \to \eta$, we have \[ \int_Z f_{\mathcal{A}}(z) d \nu_i(z) \to \int_Z f_{\mathcal{A}}(z) d \eta(z) \] which proves the lemma. \end{proof}

\section{The structure of $\mathrm{A}_\sim(\Gamma,X,\mu)$ for amenable $\Gamma$.}

When $\Gamma$ is amenable, the structure of $\mathrm{A}_\sim(\Gamma,X,\mu)$ can be completely described using the notion of an invariant random subgroup. We begin with the following, the following extends Theorem 1.8 in \cite{RTD}. Recall that if $\Gamma \curvearrowright^a (X,\mu)$ is a measure-preserving action, we have a map $\mathrm{stab}_a: X \to \mathrm{Sub}(\Gamma)$ given by $x \mapsto \mathrm{stab}_a(x)$. The type of $a$ is the invariant random subgroup of $\Gamma$ given by $(\mathrm{stab}_a)_* \mu$.

\begin{proposition}\label{prop8} If $\Gamma$ is amenable and $a,b \in \mathrm{A}(\Gamma,X,\mu)$ then $\mathrm{type}(a) = \mathrm{type}(b)$ if and only if $a \sim b$. \end{proposition}

\begin{proof} By \cite{AbEl} type is an invariant of weak equivalence so suppose $\mathrm{type}(a) = \mathrm{type}(b)$. \\
\\
Let $X^a_\infty = \{x \in X: [\Gamma: \mathrm{stab}_a(x) = \infty] \}$ and $X^b_\infty = \{x \in X: [\Gamma: \mathrm{stab}_b(x)] = \infty \}$. Notice that $X^a_\infty$ is $a$-invariant and $X^b_\infty$ is $b$-invariant and since $\mathrm{type}(a) = \mathrm{type}(b)$, $\mu(X^a_\infty) = \mu(X^b_\infty)$. Suppose that $\mu(X^a_\infty) > 0$ and let $a_\infty = a \upharpoonright X^a_\infty$ with normalized measure $\frac{\mu \upharpoonright X^a_\infty}{\mu(X^a_\infty)}$ and define $b_\infty$ similarly. Then $\mathrm{type}(a_\infty) = \mathrm{type}(b_\infty)$ and these are concentrated on the infinite index subgroups of $\Gamma$, therefore $a_\infty \sim b_\infty$ by Theorem 1.8 (2) in \cite{RTD}. Thus to prove the proposition it suffices to show the following. Note that for this we do not require $\Gamma$ to be amenable.

\begin{lemma} Suppose $a,b \in \mathrm{A}(\Gamma,X,\mu)$ are actions such that $\mathrm{type}(a) = \mathrm{type}(b)$ and these are concentrated on the finite-index subgroups of $\Gamma$. Then $a \sim b$.  \end{lemma}

\begin{proof} We may assume that $\theta = \mathrm{type}(a) = \mathrm{type}(b)$ is concentrated on the subgroups of index $n$ for some fixed $n$. Consider an $a$-orbit $C$. For each linear ordering $<_C^i$ of $C$, we get a homomorphism $\psi_C^i: \Gamma \to \mathrm{Sym}(n)$, where $\mathrm{Sym}(n)$ is the symmetric group on $n$ letters. Place a Borel linear order $\sqsubset$ on $\mathrm{Sym}(n)^\Gamma$. Let then $<^a_C = <_C^{i_0}$ be the linear order such that $\psi_C^{i_0}$ is $\sqsubset$-least among all the $\psi^i_C$. Write $\phi^a_C$ for $\psi^{i_0}_C$. Use this same construction to choose homomophisms $\phi^b_D$ for each $b$-orbit $D$. Write $\phi^a_x$ for $\phi^a_{[x]_{E_a}}$ and similarly $\phi^b_x$ for $\phi^b_{[x]_{E_b}}$.\\
\\
For a homomorphism $\phi: \Gamma \to \mathrm{Sym}(n)$ let $j_\phi$ be the corresponding action of $\Gamma$ on $\{1,\ldots,n\}$. Say $\phi$ is transitive if $j_{\phi}$ is transitive. Each transitive homomorphism $\phi: \Gamma \to \mathrm{Sym}(n)$ determines a conjugacy class $\mathcal{H}_\phi$ of index $n$ subgroups of $\Gamma$ as the stabilizers of $j_\phi$. For each $a$-orbit $[x]_{E_a}$ the stabilizers of the action of $\Gamma$ on $[x]_{E_a}$ also determine a conjugacy class $\mathcal{H}^a_x$ of index $n$ subgroups of $\Gamma$. Let $c$ be the action of $\mathrm{Sym}(n)$ on $\mathrm{Sym}(n)^\Gamma$ by $(f \cdot \phi)(\gamma)(k) = f \phi(\gamma) f^{-1}(k)$. Then $[\phi^a_x]_{E_c} = \left \{ \psi^i_{[x]_{E_a}}: <_{[x]_{E_a}}^i \mbox{ is a linear ordering of } [x]_{E_a} \right \}$. Let $\mathcal{L}$ be the set of all transitive homomorphisms $\phi: \Gamma \to \mathrm{Sym}(n)$ such that $\phi$ is $\sqsubset$-least in $[\phi]_{E_c}$. It is clear that for $\phi \in \mathcal{L}$, $\phi^a_x = \phi$ if and only if $\mathcal{H}^a_x = \mathcal{H}_\phi$. Similarly $\phi^b_x = \phi$ if and only if $\mathcal{H}^b_x = \mathcal{H}_\phi$. Thus for any $A \subseteq \mathcal{L}$, we have \begin{align*} \mu( \{x: \phi^a_x \in A \}) & = \mu( \{x: \mathcal{H}^a_x = \mathcal{H}_\phi \mbox{ for some } \phi \in A \}) \\
& = \mu( \{x: \mathrm{stab}_a(x) \mbox{ is conjugate to an element of } \mathcal{H}_\phi \mbox{ for some } \phi \in A \} ) \\ & = \theta(\{H \in \mathrm{Sub}(\Gamma): H \mbox{ is conjugate to an element of } \mathcal{H}_\phi \mbox{ for some }\phi \in A \}) \\ & = \mu( \{ x: \mathrm{stab}_b(x) \mbox{ is conjugate to an element of } \mathcal{H}_\phi \mbox{ for some } \phi \in A \} ) \\ & = \mu( \{x : \phi^b_x \in A\}). \end{align*}

Now, fix a finite set $F \subseteq \Gamma$ and a partition $A_1,\ldots,A_m$ of $X$. For each map $\omega: F \to \mathrm{Sym}(n)$ let $X^a_\omega = \{x \in X: \phi^a_x \upharpoonright F = \omega \}$ and similarly $X^b_\omega = \{x \in X: \phi^b_x \upharpoonright F = \omega \}$. Then $(X^a_{\omega})_{\omega \in \mathrm{Sym}(n)^F}$ and $(X^b_\omega)_{\omega \in \mathrm{Sym}(n)^F}$ are finite decompositions of $X$ with $\mu(X^a_\omega) = \mu(X^b_\omega)$. For $k \leq n$ let \[ X^a_{\omega,k} = \left \{x \in X^a_{\omega}: x \mbox{ is in the } k \mbox{-position with respect to }<^a_{[x]_{E_a}} \right \} \] and define $X^b_{\omega,k}$ similarly. We claim that for each $k$ there is a measure-preserving bijection $S^a_{\omega,k}$ of $X^a_{\omega,k}$ with $X^a_{\omega,1}$. Let $\sqsubset_1$ be a wellordering of $\Gamma$. For each $\gamma \in \Gamma$ let \[X^a_{\omega,k,\gamma} = \left \{x \in X^a_{\omega,k}: \mbox{ the } \sqsubset_1\mbox{- least }\delta \in \Gamma \mbox{ with  }\delta^a x \in X^a_{\omega,1} \mbox{ is equal to } \gamma \right \}. \]  Put then $S^a_{\omega,k} \upharpoonright X^a_{\omega,k,\gamma} = \gamma^a$. In particular, this shows that $\mu(X^a_{\omega,k}) = \frac{\mu(X^a_\omega)}{n}$. We can perform the same construction for $b$ and we see that $\mu(X^b_{\omega,k}) = \frac{\mu(X^b_\omega)}{n}$. So $\mu(X^a_{\omega,1}) = \mu(X^b_{\omega,1})$ and hence there is a measure-preserving bijection $T_{\omega,1}$ of each $X^a_{\omega,1}$ with $X^b_{\omega,1}$. Define a measure-preserving bijection $T_{\omega}$ of $X^a_\omega$ with $X^b_\omega$ by letting $T_{\omega}(x) = (S^b_{\omega,k})^{-1} T S^a_{\omega,k}(x)$ for $x \in X^a_{\omega,k}$. Let then $T = \bigcup_{\omega \in \mathrm{Sym}(n)^F } T_\omega$ so $T \in \mathrm{Aut}(X,\mu)$.\\
\\
We claim that for all $\gamma \in F$ and all $x \in X$, we have $T(\gamma^a x) = \gamma^b T(x)$. Indeed, suppose $x \in X^a_{\omega,k}$ so that $x$ is in the $k$-position with respect to $<^a_{[x]_{E_a}}$. Then $\gamma^a x$ is in the $\phi^a_x(\gamma)(k) = \omega(k)$ position with respect to $<^a_{[x]_{E_a}}$ so $T(\gamma^a_x)$ is in the $\omega(k)$ position of the $E_b$-class $D$ such that $T_{\omega,1} S^a_{\omega,k}(x) \in D$, where $D$ has the canonical order $<^b_D$. On the other hand, $T(x) = T_{\omega}(x)$ is in the $k$-position of $D$ with respect to $<^b_D$. Hence $\gamma^b T(x)$ is in the $\phi^b_{T(x)}(\gamma)(k) = \omega(k)$ position of $D$ and we have the claim. Now, for $i \leq m$ putting $B_i = T(A_i)$ we have for any $\gamma$ in $F$ and $i,j \leq m$, \begin{align*} \mu(\gamma^b B_i \cap B_j) & = \mu(\gamma^b T(A_i) \cap T(A_j)) \\ & = \mu(T(\gamma^a A_i) \cap T(A_j) \\ & = \mu(T(\gamma^a A_i \cap A_j)) \\ & = \mu(\gamma^a A_i \cap A_j) \end{align*} and therefore $a \sim b$. \end{proof}
 \end{proof}

In \cite{RTD}, Tucker-Drob shows that for amenable $\Gamma$, the space $\mathrm{A}_{\sim_s}(\Gamma,X,\mu)$ of stable weak equivalence classes is homeomorphic to the space $\mathrm{IRS}(\Gamma)$ of invariant random subgroups of $\Gamma$. Indeed, $\mathrm{type}(a) = \mathrm{type}(b)$ if and only if $a \sim_s b$ and the map $A_{\sim_{s}}(\Gamma,X,\mu) \to \mathrm{IRS}(\Gamma)$ given by $a \mapsto \mathrm{type}(a)$ is a homeomorphism. So we have the following.

\begin{corollary} For amenable $\Gamma$, $a \sim_s b$ if and only if $a \sim b$. \end{corollary}

Moreover, let $x \in X$, $t \in [0,1]$ and $a,b \in \mathrm{A}(\Gamma,X,\mu)$ and consider the action $ta + (1-t)b$ on $t X_1 \sqcup (1-t) X_2$. We have $\mathrm{stab}_{ta + (1-t)b} = \mathrm{stab}_a(x)$ if $x \in X_1$ and $\mathrm{stab}_b(x)$ if $x \in X_2$. Thus for any $H \leq \Gamma$, $\{x: \mathrm{stab}_{ta + (1-t)b}(x) = H\} = \{x \in X_1: \mathrm{stab}_a(x) = H \} \sqcup \{x \in X_2: \mathrm{stab}_b(x) = H\}$ so for any $A \subseteq \mathrm{Sub}(\Gamma)$ we have \begin{align*} (t\mu_1+(1-t)\mu_2)( \{ x: \mathrm{stab}_{ta+(1-t)b}(x)  \in A\}) &= (t\mu_1 + (1-t)\mu_2)( \{x \in X_1: \mathrm{stab}_a(x) \in A \} \\& \sqcup \{x \in X_2: \mathrm{stab}_b(x) \in A\}) \\ & = t \mu(\{x: \mathrm{stab}_a(x) \in H \}) + (1-t) \mu(\{x: \mathrm{stab}_b(x) \in A\}). \end{align*} Therefore $\mathrm{type}(ta+(1-t)b) = t( \mathrm{type}(a)) + (1-t)(\mathrm{type}(b))$ and Theorem \ref{thm1} follows. Note in particular that if $\Gamma$ is amenable then $ta + (1-t)a \sim a$, so for amenable groups $\mathrm{A}_\sim(\Gamma,X,\mu)$ is actually a convex space, not just a weak convex space.\\
\\
It is known (see for example \cite{EisGlas}) that $\mathrm{IRS}(\Gamma)$ is a simplex in $C(\mathrm{Sub}(\Gamma))^*$, the dual of the Banach space $C(\mathrm{Sub}(\Gamma))$ of continuous functions on $\mathrm{Sub}(\Gamma)$. So by the classical Krein-Milman theorem we have that for amenable $\Gamma$, $\cch(\ex(\mathrm{A}_\sim(\Gamma,X,\mu))) = \mathrm{A}_\sim(\Gamma,X,\mu)$. We will prove an analogous result for general $\Gamma$ using other means. Moreover, $\ex(\mathrm{IRS}(\Gamma))$ is precisely the ergodic measures in $\mathrm{IRS}(\Gamma)$ so when $\Gamma$ is amenable, $\ex(\mathrm{A}_\sim(\Gamma,X,\mu))$ is the set of actions with ergodic type.

\section{The structure of $\mathrm{A}_\sim(\Gamma,X,\mu)$ for general $\Gamma$.}

Recall from \cite{K} that $E_0$ is the equivalence relation given by eventual equality on $2^{\mathbb{N}}$ and if $E$ is an equivalence relation on $X$ and $F$ is an equivalence relation on $Y$ then a Borel homomorphism from $E$ to $F$ is a Borel map $f: X \to Y$ such that $x_1 E x_2$ implies $f(x_1) F f(x_2)$. A equivalence relation $E$ on a measure space is said to be strongly ergodic (or $E_0$-ergodic) if for any homomorphism from $E$ to $E_0$, the preimage of some $E_0$-class is conull. By Proposition 5.6 in \cite{CKTD} if $a$ is strongly ergodic then every $b$ with $b \sim a$ is ergodic. In particular, $\frac{1}{2}a + \frac{1}{2}a$ is not ergodic, so $\frac{1}{2}a + \frac{1}{2}a$ is not weakly equivalent to $a$ when $a$ is strongly ergodic. By Theorem 1.2 in \cite{KTsan}, the Bernoulli shift $\Gamma \curvearrowright ([0,1]^\Gamma,\lambda^\Gamma)$ with $\lambda$ Lebesgue measure on $[0,1]$ is strongly ergodic when $\Gamma$ is nonamenable. Thus when $\Gamma$ is nonamenable, $\mathrm{A}_\sim(\Gamma,X,\mu)$ is not a convex space, only a weak convex space. We now prove Theorem \ref{thm2}

\begin{proof} \textbf{\textbf{(of Theorem \ref{thm2})}} Write $A = \mathrm{A}_\sim(\Gamma,X,\mu)$. Let $B = \cch(\ex(A))$ and suppose toward a contradiction that there exists $x \in A \setminus B$. Since $B$ is compact, $d(x,B) > 0$. Let $\alpha = \sup_{y \in A} d(y,B)$ and let $C = \{y \in A: d(y,B) = \alpha\}$. Then $C$ is nonempty, disjoint from $B$ and $C$ is a face of $A$.\\
\\
Let $\mathcal{F}$ be the family of faces of $C$, ordered by reverse inclusion. Suppose $\{F_i\}_{i \in I}$ is a linearly ordered subset of $\mathcal{F}$ and consider $\bigcap_{i \in I} F_i$. If $x,y \in C$ and $0 <t <1$ are such that $tx + (1-t)y \in \bigcap_{i \in I} F_i$, then $x,y \in F_i$ for each $i$ since each $F_i$ is a face. Hence $\bigcap_{i \in I} F_i$ is a face. It is nonempty by compactness. So Zorn's Lemma guarantees there exist minimal elements of $\mathcal{F}$. Let $F$ be such a minimal element.\\
\\
Choose $y \in F$ and suppose toward a contradiction that there exists $y' \in F$ with $y' \notin \cch(\{y\})$. Then $\cch(\{y\})$ is a compact convex set, so letting $G = \left \{z \in  F: d(z,\cch(\{y\})) = \sup_{w \in F} d(w,\cch(\{y\})) \right \}$, $G$ is a nonempty face of $F$ disjoint from $\cch(\{y\})$, contradiction the minimality of $F$. So for all $y \in F$ we have $F \subseteq \cch(\{y\})$. Fix such a $y$. Note that $\cch(\{y\}) = \ch(\{y\})$. We claim that $y$ is an extreme point of $C$. Assuming this, since $C$ is a face of $A$ we have that $y$ is an extreme point of $A$ and we have a contradiction to the hypothesis that $C \cap B = \emptyset$.\\
\\
Suppose first that there do not exist $a,b \in C$ and $0 < t <1$ such that $y = ta + (1-t)b$. Then $y$ is an extreme point of $C$ be definition. So let $a,b \in C$ and $0 < t < 1$ be such that $y = ta + (1-t)b$. We must show that $y \sim a \sim b$. Since $F$ is a face of $C$, we have $a,b \in F$. Thus we can write $a = \sum_{i=1}^n s_i y$ and $b = \sum_{i=1}^k r_i y$ for $s_i,r_i \in [0,1]$. By Proposition \ref{prop2} and associativity we have $y \sim \left( \sum_{i=1}^n t s_i y + \sum_{i=1}^k (1-t)r_i y \right)$. Since $0 < t < 1$, iterating this argument we find that for any $\delta > 0$, there is $m \in \mathbb{N}$ and $(\lambda_i)_{i=1}^m \subseteq [0,1]$ such that $\lambda_i \leq \delta$ for all $i$ and $y \sim \sum_{i=1}^m \lambda_i y$.\\
\\
We claim that this implies $y \sim \kappa y + (1-\kappa)y$ for all $\kappa \in [0,1]$. Note that $\kappa y + (1-\kappa) y$ is isomorphic to $\iota_{\kappa, 1-\kappa} \times y$, where $\iota_{\kappa, 1- \kappa}$ is the trivial action of $\Gamma$ on $(\{0,1\},m_\kappa)$ where $m_\kappa( \{0\}) = \kappa$ and $m_\kappa( \{1\}) = 1-\kappa$. Hence $y$ is a factor of $\kappa y + (1-\kappa)y$ and it thus suffices to show $\kappa y + (1-\kappa)y \prec y$.\\
\\
Let $X_1,X_2$ be two copies of $X$, let $n,k \in \mathbb{N}$, $\epsilon > 0$ and a partition $\mathcal{P} = (P_i)_{i=1}^k$ of $X_1 \sqcup X_2$ be given. As before, we get a partition $\mathcal{P}_1 = (P^1_i)_{i=1}^k$ with $P^1_i = P_i \cap X_1$ of $X_1$ and similarly a partition $\mathcal{P}_2 = (P^2_i)_{i=1}^k$ with $P^2_i = P_i \cap X_2$ of $X_2$. Now, choose $\delta < \frac{\epsilon}{2}$. Then we can find $m$ and $(\lambda_p)_{p=1}^m$ such that $y \sim \sum_{p=1}^m \lambda_p y$ and for some $l \leq m$ we have $\kappa - \frac{\epsilon}{2} \leq \sum_{p=1}^l \lambda_p \leq \kappa$. Let now $X'_p$ be a copy of $X$ for each $p \leq m$, and for $q \in \{0,1\}$ let $P^q_{i,p}$ be the corresponding copy of $P^q_i$ sitting in $X'_p$. Let $\mathcal{Q} = (Q_i)_{i=1}^k$ be the partition of $\bigsqcup_{p=1}^m X'_p$ given by $Q_i = \left( \bigsqcup_{p=1}^l P^1_{i,p} \right) \sqcup \left( \bigsqcup_{p= l+1}^m P^2_{i,p} \right)$. Then for $s \leq n$ and $i,j \leq k$ we have 

\begin{align*} \Bigg \vert (\kappa \mu & + (1-\kappa) \mu)( \gamma_s^{\kappa y + (1-\kappa) y} P_i \cap P_j) - \left ( \sum_{p=1}^m \lambda_p \mu \right ) \left( \gamma_s^{\sum_{p=1}^m \lambda_p y} Q_i \cap Q_j \right) \Bigg \vert\\ & \leq \left \vert \kappa \mu(\gamma_s^y P_i^1 \cap P_j^1) - \left( \sum_{p=1}^l \lambda_p \mu (\gamma_s^y P^1_{i,p} \cap P^1_{j,p} ) \right) \right \vert +  \left \vert (1- \kappa) \mu(\gamma_s^y P_i^2 \cap P_j^2) - \left( \sum_{p=l+1}^m \lambda_p \mu (\gamma_s^y P^2_{i,p} \cap P^2_{j,p} ) \right) \right \vert \end{align*} \begin{align*} & = \left \vert \kappa \mu(\gamma_s^y P_i^1 \cap P_j^1) - \left( \sum_{p=1}^l \lambda_p \right) \mu (\gamma_s^y P^1_i \cap P^1_j ) \right \vert +  \left \vert (1-\kappa) \mu(\gamma_s^y P_i^2 \cap P_j^2) - \left( \sum_{p=l+1}^m \lambda \right) \mu (\gamma_s^y P^2_i \cap P^2_j )\right \vert \\ &= \left \vert \left ( \kappa - \sum_{p=1}^l \lambda_p \right) \mu( \gamma^y_s P^1_i \cap P^1_j) \right \vert + \left \vert \left( (1-\kappa) - \sum_{p=l+1}^m \lambda_p \right) \mu(\gamma_s^y P^2_i \cap P^2_j) \right \vert \\ & \leq \left \vert \left ( \kappa - \sum_{p=1}^l \lambda_p \right)  \right \vert + \left \vert \left( (1-\kappa) - \sum_{p=l+1}^m \lambda_p \right) \right \vert \leq \frac{\epsilon}{2} + \frac{\epsilon}{2} = \epsilon. \end{align*}

Since $y \sim \sum_{p=1}^m \lambda_p y$, $\kappa y + (1-\kappa)y \prec y$ and we are done. \end{proof}

We note that a metrizable topological vector space $V$ is locally convex if and only if its topology is induced by a countable family of seminorms $ \left(| \cdot |^V_n \right)_{n=1}^\infty$. Then $p(v,w) = \sum_{n=1}^\infty \frac{1}{2^n} | v - w |^V_n$ is a compatible metric on $V$, which is easily seen to obey Lemma \ref{lem3}. Thus the technique used to prove Theorem \ref{thm2} works to prove the metrizable case of the classical Krein-Milman theorem using only the convex and metric structure of $V$, not the vector space structure in the form of linear functionals.\\
\\
Before proving Theorem \ref{thm4}, we briefly discuss the ergodic decomposition in the context of weak equivalence classes. Suppose $a \in \mathrm{A}(\Gamma,X,\mu)$ and $a = \int_Z a_z d \eta(z)$ is the ergodic decomposition of $a$, that is to say we have a factor map $\pi: (X,\mu) \to (Z,\eta)$ such that if $\mu = \int_Z \mu_z d \eta(z)$ is the disintegration of $\mu$ over $(Z,\eta)$ via $\pi$ then $\mu_z( \pi^{-1}(z) ) = 1$ and $\Gamma \curvearrowright^a (\pi^{-1}(z),\mu_z)$ is isomorphic to $a_z$. Furthermore, the assignment $z \mapsto \mu_z$ from $(Z,\eta) \to M_a(X)$ is Borel, where $M_a(X)$ is the space of $a$-invariant probablity measures on $X$ (we may assume here that $X$ is a Polish space). Recall that $A^*_\sim(\Gamma)$ is the space of weak equivalence classes of all measure-preserving actions of $\Gamma$, including those actions on finite space. $A^*_\sim(\Gamma)$ is topologized using the exact same metric as we use to topologize $\mathrm{A}_\sim(\Gamma,X,\mu)$. We would like to conclude that the assignment $z \mapsto [a_z]$ is measurable from $(Z,\eta)$ to $A^*_\sim(\Gamma)$, where $[a_z]$ is the weak equivalence class of $a_z$. This is a consequence of the following lemma.

\begin{lemma} \label{lem10} Let $\Gamma \curvearrowright^a Y$ be a Borel action of $\Gamma$ on a Polish space $Y$. Then the map $\Theta$ from $M_a(Y)$ to $A^*_\sim(\Gamma)$ given by $\nu \mapsto [a_\nu]$ is Borel, where $[a_\nu]$ is the weak equivalence class of the measure preserving action $a_\nu = \Gamma \curvearrowright^a (Y,\nu)$. \end{lemma}

\begin{proof} Fix a measure $\nu \in M_a(Y)$ and consider $\Theta^{-1}(U)$ where \[U = \{ [a] \in A_{\sim}^*(\Gamma): d_H( C_{n,k}(a_\nu),C_{n,k}(a)) < \epsilon \mbox{ for all }n,k \leq N \} \] for some $N \in \mathbb{N}$ and $\epsilon > 0$, so $U$ is a basic open neighborhood of $\Theta(\nu) = a_\nu$. Since \[ U = \bigcup_{m=1}^\infty \bigcap_{n,k =1}^N \left \{ [b] \in A_{\sim}^*(\Gamma): d_H(C_{n,k}(a_\nu), C_{n,k}(b) ) \leq \epsilon - \frac{1}{m}  \right\}, \] it suffices to show $\Theta^{-1}(V)$ is Borel for a set $V$ of the form \[V = \{ [b] \in A_{\sim}^*(\Gamma): d_H( C_{n,k}(a_\nu), C_{n,k}(b) ) \leq r \}. \] Fixing $n$ and $k$ we write $C(b)$ for $C_{n,k}(b)$. Now, let $K$ and $L$ be compact subsets of a compact Polish space $W$ with metric $p$, let $D_K$ be dense in $K$ and $D_L$ be dense in $L$. We have \begin{align*} d_H(K,L) \leq r & \iff \max_{x \in K} \inf_{y \in L} p(x,y) \leq r \mbox{ and } \max_{y \in L} \inf_{x \in K} p(y,x) \leq r \\ & \iff ( \forall x \in K )(\forall \delta > 0 )( \exists y \in L)( p(x,y) < r + \delta )\\  &\hspace{2 in} \wedge ( \forall y \in L )( \forall \delta >  0)( \exists x \in K)( p(y,x) < r + \delta ) \\ & \iff (\forall x \in D_K)( \forall \delta > 0)( \exists y \in D_L)( p(x,y) < r + \delta) \\ & \hspace{2 in}\wedge (\forall y \in D_L)(\forall \delta > 0) (\exists y \in D_L)( p(y,x) < r + \delta ) \end{align*}  
If $\mathbf{L}$ is a countable algebra generating the Borel $\sigma$-algebra $\mathbf{B}(Y)$ of $Y$, then $\mathbf{L}$ is dense in $\mathrm{MALG}(Y,\rho)$ for any Borel probability measure $\rho$ on $Y$. Regarding a partition of $Y$ into $k$ pieces as an element of $\mathbf{B}(Y)^N$ and considering $\mathbf{L}^k$, we see that there exists a fixed countable family $(\mathcal{A}_m)_{m=1}^\infty$ of partitions of $Y$ such that for any Borel probability measure $\rho$ on $Y$, $(\mathcal{A}_m)_{m=1}^\infty$ is dense in the set of $k$-partitions of $X$ with topology inherited from $\mathrm{MALG}(Y,\rho)$. We may further assume that each element of each $\mathcal{A}_m$ is clopen. This implies that the set $(M^{\mathcal{A}_m}(a_\rho))_{m=1}^\infty$ is dense in $C(a_\rho)$ for any Borel probability measure $\rho$. Therefore we have \begin{align*} V & = \left( \bigcap_{m=1}^\infty \bigcap_{l = 1}^\infty \bigcup_{i=1}^\infty \left \{ b \in A^*_\sim(\Gamma):  \left \vert M^{\mathcal{A}_m}(a_\nu) - M^{\mathcal{A}_i}(b) \right \vert < r + \frac{1}{l} \right \} \right) \\  &\hspace{1 in}  \cap \left( \bigcap_{m=1}^\infty \bigcap_{l = 1}^\infty \bigcup_{i=1}^\infty \left \{ b \in A^*_\sim(\Gamma):  \left \vert M^{\mathcal{A}_i}(a_\nu) - M^{\mathcal{A}_m}(b) \right \vert < r + \frac{1}{l} \right \} \right). \end{align*}
Now, $|M^{\mathcal{A}_i}(a_\nu) - M^{\mathcal{A}_m}(a_{\rho})| < s$ if and only if $|\nu(\gamma^a A_i^j \cap A_i^t) - \rho( \gamma^a A_m^i \cap A_m^t)| <s$ for all $A_i^j,A_i^t \in \mathcal{A}_i$ and $A_m^i,A_m^t \in \mathcal{A}_m$. Since for any pair $J_1,J_2 \subseteq Y$ the set $\{\rho: |\nu(J_1) - \nu(J_2)| < s\}$ is Borel, we see \[\Theta^{-1} \left( \left \{ b \in A^*_\sim(\Gamma):  \left \vert M^{\mathcal{A}_i}(a_\nu) - M^{\mathcal{A}_m}(b) \right \vert < r + \frac{1}{l} \right \} \right) \] is Borel and consequently $\Theta^{-1}(V)$ is Borel. \end{proof}

We now prove Theorem \ref{thm4}

 \begin{proof}\textbf{\textbf{(of Theorem \ref{thm4})}} Let $\Theta: Z \to A^*_\sim(\Gamma)$ be the map sending each point in $z$ to the weak equivalence class $[a_z]$, so $\Theta$ is measurable by Lemma \ref{lem10}. Suppose towards a contradiction that the theorem fails. Then for every set $Z' \subseteq Z$ with $\eta(Z') = 1$, there is more than one weak equivalence class in the set $\{[a_z]: z \in Z' \}$. Equivalently, the measure $\Theta_* \eta$ on $A^*_\sim(\Gamma)$ is not supported on a single point. We can thus split $A^*_\sim(\Gamma)$ into two disjoint sets $Y_1,Y_2$ such that $0 < \Theta_*\eta(Y_1), \Theta_*\eta(Y_2) <1$. Letting $A_i = \Theta^{-1}(Y_i)$, we get disjoint measurable sets $A_1,A_2 \subseteq Z$ such that $0 < \eta(A_1), \eta(A_2) < 1$ and for all $z \in A_1$ and all $w \in A_2$ we have that $z \nsim w$.\\
\\
Recall that for a measure-preserving action $b$ of $\Gamma$ and $n,k \in \mathbb{N}$ the set $C_{n,k}(a) \subseteq [0,1]^{n \times k \times k}$ was defined in Section \ref{sec3}. 

\begin{lemma}\label{lem9} For any action $b$ of $\Gamma$ on a probability space $(Y,\nu)$, we have $\cch(C_{n,k}(b)) \subseteq C_{n,k}(\iota \times b)$. \end{lemma}

\begin{proof} Write $C_{n,k}(b) = C(b)$. Suppose $x \in \cch(C(b))$. Then we can find points $(x_i)_{i=1}^\infty$ such that $\lim_{i \to \infty} x_i = x$ and each $x_i$ has the form $x_i = \sum_{j=1}^{j(i)} \alpha_i^j x_i^j$ for $(x_i^j)_{j=1}^{j(i)} \subseteq C(b)$ and $(\alpha_i^j)_{j=1}^{j(i)} \subseteq [0,1]$ with $\sum_{j=1}^{j(i)} \alpha^i_j = 1$ for each $i$. Without loss of generality we may assume that each $x_i^j$ has the form $M^{\mathcal{A}_i^j}(b)$ for a partition $\mathcal{A}_i^j = (A_{i,l}^j)_{l=1}^k$ of $Y$ into $k$ pieces. Fixing $i$ consider the action $\sum_{j=1}^{j(i)} \alpha_i^j b$ on the space $\left( \bigsqcup_{j=1}^{j(i)} Y_j, \sum_{j=1}^{j(i)} \alpha_i^j \nu_j \right)$, where each $(Y_j,\nu_j)$ is a copy of $(Y,\nu)$. Let $\mathcal{B} = (B_l)_{l=1}^k$ be the partition of $\bigsqcup_{j=1}^{j(i)} Y_j$ given by letting $B_l = \bigsqcup_{j=1}^{j(i)} A_{i,l}^j$, where $A_{i,l}^j$ sits inside the $j$ copy of $Y$. For any $p \leq n$ and $l,m \leq k$ and $x \in [0,1]^{n \times k \times k}$ let $(x)_{p,l,m}$ be the $p,l,m$ coordinate of $x$. We then have

\begin{align*} \left( M^{\mathcal{B}} \left ( \sum_{j=1}^{j(i)} \alpha_i^j b \right) \right)_{p,l,m} &= \left( \sum_{j=1}^{j(i)} \alpha_i^j \nu_j \right) \left ( \gamma_p^{\sum_{j=1}^{j(i)} \alpha_i^j b} B_l \cap B_m \right) \\ & = \sum_{j=1}^{j(i)} \left( \alpha_i^j \nu_j (\gamma_p^b A_{i,l}^j \cap A_{i,m}^j)\right) \\ & = \sum_{j=1}^{j(i)} \alpha_i^j \left( M^{\mathcal{A}_i^j}(b) \right)_{p,l,m} \end{align*}

Therefore \[ M^{\mathcal{B}} \left ( \sum_{j=1}^{j(i)} \alpha_i^j b \right) = \sum_{j=1}^{j(i)} \alpha_i^j \left( M^{\mathcal{A}_i^j}(b) \right) = x_i. \]

We have shown that $x_i \in C \left( \sum_{j=1}^{j(i)} \alpha_i^j b \right)$. Since $\sum_{j=1}^{j(i)} \alpha_i^j b$ is a factor of $b \times \iota$, we have $x_i \in C(b \times \iota)$. Since $\lim_{i \to \infty} x_i = x$ and $C(b \times \iota)$ is closed, the lemma follows. \end{proof}

It is clear that for any two measure-preserving actions $b,c$ we have $b \prec c$ if and only if $C_{n,k}(b) \subseteq C_{n,k}(c)$ for all $n,k$. We claim that there are disjoint subsets $A_3, A_4 \subseteq Z$ of positive measure such that for some pair $n_0,k_0$, every $z \in A_3$ and every $w \in A_4$ we have $C_{n_0,k_0}(a_z) \nsubseteq \cch(C_{n_0,k_0}(a_w))$. For $z \in A_3$ let $R_z = \{ w \in A_2: a_z \nprec a_w \}$. Since $a_z$ is ergodic, $a_z \prec a_w \times \iota$ implies $a_z \prec a_w$. Therefore $R_z = \{w \in A_2: a_z \nprec a_w \times \iota \}$. \\
\\
Assume first that there is a set $D_3 \subseteq A_1$ with $\eta(D_3) > 0$ such that for each $z \in D_3$ we have $\eta(R_z) > 0$. Write $\hat{K}$ for $\cch(K)$. By Lemma \ref{lem9} we can write $R_z = \bigcup_{n,k = 1}^\infty R_z^{n,k}$ where $R_z^{n,k} =  \left \{ w \in A_2: C_{n,k}(a_z) \nsubseteq \widehat{C_{n,k}(a_w)} \right \}$. Thus for each $z$ there is a lexicographically least pair $(n_z,k_z)$ such that $\eta(R_z^{n_z,k_z}) > 0$. Therefore there is a pair $n_0,k_0$ and a set $D_4 \subseteq D_3$ such that $\eta(D_4) > 0$ and for all $z \in D_4$ we have $\eta(R_z^{n_0,k_0}) > 0$. Fixing $n_0$ and $k_0$ we write $C(b)$ for $C_{n_0,k_0}(b)$. Let $(w_j)_{j=1}^\infty \subseteq A_2$ be a sequence of points such that the family $\left(\widehat{C(a_{w_j})} \right)_{j=1}^\infty$ is dense in the space $\left \{ \widehat{C(a_w)}: w \in A_2 \right\}$ with respect to the Hausdorff metric $d_H$ on the space on compact subsets of $[0,1]^{n_0 \times k_0 \times k_0}$. Let then $F_{j,l} = \left\{ w \in A_2: d_H\left( \widehat{C(a_w)},\widehat{C(a_{w_j})} \right) < \frac{1}{l} \right\}$.\\
\\
Fix $z \in D_4$ and choose $w \in R_z^{n_0,k_0}$. By hypothesis there is $\epsilon > 0$ such that $C(a_z) \nsubseteq B_{\epsilon} \left(\widehat{C(a_w)} \right)$, where $B_{\epsilon}(K)$ denotes the ball of radius $\epsilon$ around $K$. Then if we choose $j$ so that $d_H \left(\widehat{C(a_{w_j})},\widehat{C(w)} \right) < \frac{\epsilon}{2}$ and $l$ so that $\frac{1}{l} < \frac{\epsilon}{2}$ we have $w \in F_{j,l} \subseteq R_z^{n_0,k_0}$. Hence there is a subset $\mathcal{J} \subseteq \mathbb{N}^2$ such that $R_z^{n_0,k_0} = \bigcup_{(j,l) \in \mathcal{J}} F_{j,l}$. So for each $z$ we can choose a lexicographically least pair $(j_z,l_z)$ such that $\eta(F_{j_z,l_z}) > 0$ and $F_{j_z,l_z} \subseteq R_z^{n_0,k_0}$. There is then a pair $(j_0,l_0)$ and a set $E_3 \subseteq D_3$ with $\eta(E_3) > 0$ such that $\eta(F_{j_0,l_0}) > 0$ and for all $z \in E_3$ and all $w \in F_{j_0,l_0}$ we have $C(a_z) \nsubseteq \widehat{C(a_w)}$. So take $A_3 = E_3$ and $A_4 = F_{j_0,l_0}$. Thus we are left with the case $\eta(R_z) = 0$ for almost all $z \in A_1$. Then for almost all $w \in A_2$ and almost all $z \in A_1$ we must have $a_w \nprec a_z$, so a symmetric argument gives the claim.\\
\\
Given a (real) topological vector space $V$, we say a hyperplane in $V$ is a set of the form $H_{\ell,\alpha} = \{v \in V: \ell(v) = \alpha \}$ for some continuous linear functional $\ell$ and $\alpha \in \mathbb{R}$. Given disjoint compact subsets $W_1,W_2 \subseteq V$ we say that $H_{\ell,\alpha}$ separates $W_1$ from $W_2$ if $W_1 \subseteq \{v \in V: \ell(v) < \alpha \}$ and $W_2 \subseteq \{v \in V: \ell(v) > \alpha \}$.

\begin{lemma} Let $S \subseteq \mathbb{R}^n$ be compact. Then there is a countable family $(H_i)_{i=1}^\infty$ of hyperplanes such that for any $x \in S$ and any compact convex $W \subseteq S$ there is $i$ so $H_i$ separates $\{x\}$ from $W$. \end{lemma}

\begin{proof} Let $(\ell_j)_{j=1}^\infty$ be a countable set of linear functionals which is dense in the $\sup$ norm on $S$. Enumerate $\mathbb{Q}$ as $(q_m)_{m=1}^\infty$ and let $H_{j,m} = \{s \in S: \ell_j(s) = q_m\}$. Given $x$ and $W$, by Hahn-Banach find a linear functional $\ell$ and $\alpha \in \mathbb{R}$ so that $H = H_{\ell,\alpha}$ separates $x$ from $W$. Let $r = \min \left( \inf_{h \in H} ||x -h ||,  \inf_{ \substack{ h \in H, \\ w \in W}} || h - w|| \right)$ so $r > 0$. Then choose $m$ so $|q_m - \alpha| < \frac{r}{2}$ and $j$ so $\sup_{s \in S} |\ell(s) - \ell_j(s)| < \frac{r}{2}$. Then $H_{j,m}$ separates $x$ from $W$. \end{proof}

Now take $S = [0,1]^{n_0 \times k_0 \times k_0}$ and fix a family $(H_i)_{i=1}^\infty$ of hyperplanes as in the lemma. Since $\widehat{C(a_w)}$ is compact convex for each $w \in A_4$ and for all $z \in A_3$ we have $C(a_z) \nsubseteq \widehat{C(a_w)}$, for each pair $(z,w) \in A_3 \times A_4$ there is an index $i(z,w)$ and a point $x_{z,w} \in C(a_z)$ such that $H_{i(z,w)}$ separates $x_{z,w}$ from $\widehat{C(a_w)}$. Fix $z \in A_3$. Taking $(w_j)_{j=1}^\infty$ as before, for $(j,l) \in \mathbb{N}^2$ let $G_{j,l} = \left \{w \in A_4: d_H\left( \widehat{C(a_w)}, \widehat{C(a_{w_j})}\right) < \frac{1}{l} \right \}$. Choosing $w \in A_4$, let $\epsilon = d_H \left( \widehat{C(w)}, H_{i(z,w)} \right)$ so $\epsilon > 0$. Finding $j_{z,w}$ so $d_H \left( \widehat{C(w_j)}, \widehat{C(w)} \right) < \frac{\epsilon}{2}$ and $l_{z,w}$ so $\frac{1}{l} < \frac{\epsilon}{2}$ we have $w \in G_{j_{z,w},l_{z,w}}$ and $H_{i(z,w)}$ separates $x_{z,w}$ from $\widehat{C(u)}$ for all $u \in G_{j_{z,w},l_{z,w}}$. Then we have $A_4 = \bigcup_{\substack{(j_{z,w},l_{z,w}): \\w \in A_4 }} G_{j_{z,w},l_{z,w}}$ so we can find $w_0$ so that $\eta\left(G_{j_{z,w_0},l_{z,w_0}} \right) > 0$. Let then $G_z = G_{j_{z,w_0},l_{z,w_0}}$, $x_z = x_{z,w_0}$ and $i(z) = i(z,w_0)$ so that $H_{i(z)}$ separates $x_z$ from $\widehat{C(u)}$ for all $u \in G_z$. Since the $G_z$ were chosen from a countable family, we can find a set $A_5 \subseteq A_3$ of positive measure such that $G_z = G$ is the same for all $z \in A_5$. We can then find an index $i$ and a set $A_6 \subseteq A_5$ of positive measure such that for all $z \in K$, $H_i = H$ separates $x_z$ from $\widehat{C(u)}$ for all $u \in G$. $H$ splits $[0,1]^{n \times k \times k}$ into two closed convex sets $H_+$ and $H_-$, where $H_+$ contains the $x_z$ and $H_-$ contains the $C(u)$.\\
\\
For $S \subseteq Z$ with $\eta(S) > 0$ let $\eta_S = \frac{ \eta \upharpoonright S}{\eta(S)}$ be normalized measure on $S$. By Lemma \ref{lem11} we have $C \left ( \int_G a_u d \eta_G(u)  \right ) \subseteq \cch \left( \bigcup_{u \in G} C(u) \right) \subseteq H_-$. Write $A_6 = \bigcup_{p = 1}^\infty A_6^p$, where $A_6^p = \left \{ z \in A_6: d_H( x_z, H) \geq \frac{1}{p} \right \}$ and find $p$ so $\eta(A_6^p) > 0$. Letting $K = A_6^p$, for all $z \in K$, $x_z$ is an element of the closed convex set $H_+^p = \{ y \in H_+: d_H(y,H) \geq \frac{1}{p} \}$ and $H_+^p$ is disjoint from $H_-$. We have $\int_K x_z d \eta_K(z) \in C \left ( \int_K a_z d \eta_K(z) \right )$ and $\int_K x_z d \eta_K(z) \in H_+^p$. Since $C \left ( \int_G a_u d \eta_G(u)  \right ) \subseteq H_-$ we see that $C \left ( \int_K a_z d \eta_K(z) \right ) \nsubseteq C \left ( \int_G a_u d \eta_G(u)  \right )$ and it follows that $\int_K a_z d \eta_K(z) \nsim  \int_G a_u d \eta_G(u)$. Let $L_1 = K, L_2 = G$ then there is $i \in \{1,2\}$ with $\int_{L_i} a_z d \eta_{L_i}(z) \nsim a$. Since $0 < \eta(L_i) < 1$, we can write \[a = \eta(L_i) \left ( \int_{L_i} a_z d \eta_{L_i}(z) \right) + \eta( Z \setminus L_i) \left ( \int_{Z \setminus L_i} a_z d \eta_{Z \setminus L_i}(z) \right) \] which contradicts our assumption that $a$ is an extreme point.  \end{proof}

We now prove Theorem \ref{thm5}. Recall that the uniform topology on $\mathrm{Aut}(X,\mu)$ is given by the metric $d_u(T,S) = \mu(\{x: Tx \neq Sx \})$. If $\mathcal{P} = \{P_1,\ldots,P_p \}$ is a partition of a space on which $\mathbb{F}_N$ acts by an action $a$, $J \subseteq \mathbb{F}_N$ is finite and $\tau: J \to p$ let $P^a_{\tau} = \bigcap_{\gamma \in J} \gamma^a P_{\tau(\gamma)}$. 

\begin{proof} \textbf{\textbf{(of Theorem \ref{thm5})}} Let $a$ be a free action of $\mathbb{F}_N$. By replacing $a$ with $a \times \iota$ if necessary, we may assume that for each $n,k$ the set $C_{n,k}(a)$ is closed and convex. Fix integers $n_0$ and $k_0$ and $\epsilon > 0$. It is enough to find a free ergodic action $b$ of $\mathbb{F}_N$ such that for all $n \leq n_0$ and $k \leq k_0$ we have $d_H (C_{n,k}(a),C_{n,k}(b)) < \epsilon$. Let $\{ \gamma_1,\ldots,\gamma_{n_0} \} = F_0$ be the finite subset of $\mathbb{F}_N$ under consideration. Let $s = s_{\mathbb{F}_N}$ be the Bernoulli shift of $\mathbb{F}_N$ acting on $\left( 2^{\mathbb{F}_N}, \nu \right)$ where $\nu$ is the product measure. For any action $c$ of $\mathbb{F}_N$ on $(X,\mu)$ and $\gamma \in \mathbb{F}_N$ we have \[ \{ (x,y) \in X \times 2^{\mathbb{F}_N}: \gamma^{c \times s} (x,y) \neq \gamma^{a \times x}(x,y) \} = \{ x \in X: \gamma^c x \neq \gamma^a x \} \times Y \] and hence \[ (\mu \times \nu)(\{ (x,y) \in X \times 2^{\mathbb{F}_N}: \gamma^{c \times s} (x,y) \neq \gamma^{a \times x}(x,y) \}) = \mu(\{ x \in X: \gamma^c x \neq \gamma^a x \}). \] 

Assume $d_u(\gamma^a,\gamma^c) < \frac{\epsilon}{16}$ for all $\gamma \in F_0$. Then for any measurable partition $\mathcal{A} = A_1,\ldots,A_k$ of $X \times 2^{\mathbb{F}_N}$, all $\gamma \in F_0$ and all $i,j \leq k$ we have \[| (\mu \times \nu)(\gamma^{a \times s} A_i \cap A_j) - (\mu \times \nu)(\gamma^{c \times s} A_i \cap A_j)| < \frac{\epsilon}{16} \] for all $\gamma \in F_0$. In the notation of Section \ref{sec3}, $\rho \left(M_{n,k}^\mathcal{A}(a \times s),M_{n,k}^\mathcal{A}(c \times s) \right ) < \frac{\epsilon}{16}$ where $\rho$ is the supremum metric on $[0,1]^{n \times k \times k}$. Choose a finite collection $\mathscr{L}$ of measurable subsets of $X \times 2^{\mathbb{F}_N}$ such that for every measurable partition $\mathcal{A}$ of $X \times 2^{\mathbb{F}_N}$ there is a partition
$\mathcal{B} \subseteq \mathscr{L}$ such that $\rho\left( M^\mathcal{A}_{n,k}(a \times s), M^\mathcal{B}_{n,k}(a \times s) \right) < \frac{\epsilon}{16}$. Then for every such $\mathcal{A}$ there exists $\mathcal{B} \subseteq \mathscr{L}$ such that $\rho \left( M^\mathcal{A}_{n,k}(c \times s), M^\mathcal{B}_{n,k}( c \times s) \right) < \frac{3 \epsilon}{16}$. \\
\\
For $\gamma \in \mathbb{F}_N$ let $\pi_{\gamma}: 2^{\mathbb{F}_N} \to 2 $ be projection onto the $\gamma$ coordinate. For $i \in \{0,1\}$ let $S_i = \pi_e^{-1}(\{i\})$ and put $\mathcal{S} = \{S_1,S_2\}$.  Choose now a finite partition $\mathcal{R} = \{R_1,\ldots,R_r \}$ of $X$ and a finite subset $F \subseteq \mathbb{F}_N$ containing $F_0$ such that for every $A \in \mathcal{L}$ there are sets $R_j$ with $1 \leq j \leq r$ and a family of functions $(\tau_j)_{j=1}^t$ with $\tau_j: F \to 2$ such that \[ \mu \left ( \left( \bigsqcup_{j=1}^t R_j \times S^s_{\tau_j} \right) \triangle A \right) < \frac{\epsilon}{16}. \] Write $\mathcal{P} =\mathcal{R} \times \mathcal{S}$. We can identify a function $\theta: F \to r \times 2$ with a pair $(\sigma,\tau)$ where $\sigma:F \to r$ and $\tau:F \to 2$ so \[ P^{c \times s}_{\theta} = \bigcap_{\gamma \in F} \gamma^b P^{c \times s}_{\theta(\gamma)} = \left( \bigcap_{\gamma \in F}\gamma^c R_{\sigma(\gamma)} \right) \times \left( \bigcap_{\gamma \in F} \gamma^s S_{\tau(\gamma)} \right) = R^c_{\sigma} \times S^s_{\tau}.\] 

Note that for any $j \leq r$, $R_j \times S^s_{\tau}$ is a finite disjoint union of sets of the form $R^c_{\sigma} \times S^s_{\tau}$, hence any $A \in \mathscr{L}$ is within $\frac{\epsilon}{16}$ of finite disjoint union of sets of the form $P^{c \times s}_{\theta}$ for $\theta: F \to r \times 2$.\\
\\
Let $\delta = \frac{\epsilon}{4 (2r)^{2|F|}}$. Fix an ergodic action $c$ of $\mathbb{F}_N$ such that $d_u\left(\gamma^a,\gamma^c \right) < \frac{\delta^2}{32 |F|^2 (2r)^{|F|^2}}$ for all $\gamma \in F$. (For example use the fact that the ergodic automorphisms are uniformly dense in $\mathrm{Aut}(X,\mu)$ to move one of the generators $\gamma$ of $\mathbb{F}_N$ so it acts ergodically but is still sufficiently close to $\gamma^a$). Then clearly $d_H(C_{n,k}(a),C_{n,k}(c)) < \frac{\epsilon}{2}$ for all $n \leq n_0$ and $k \leq k_0$. Let $b = c \times s$. Since $c$ is ergodic and $s$ is free and mixing, $b$ is free and ergodic. Thus it is sufficient to show $d_H(C_{n,k}(c),C_{n,k}(b)) < \frac{\epsilon}{2}$ for all $n \leq n_0$, $k \leq k_0$. Since $c \prec b$, it is sufficient to show that for every partition $\mathcal{A}$ of $X \times 2^{\mathbb{F}_N}$ there is a partition $\mathcal{C}$ of $X$ such that $\rho \left(M^\mathcal{A}_{n,k}(b),M^\mathcal{C}_{n,k}(c) \right) < \frac{\epsilon}{2}$. By our previous reasoning, for each partition $\mathcal{A} = (A_1,\ldots,A_k)$ of $X \times 2^{\mathbb{F}_N}$ there is a partition $\mathcal{B}$ whose pieces are disjoint unions of sets of the form $P^b_\theta$ for $\theta: F \to r \times 2$ such that $\rho \left( M^{\mathcal{A}}_{n,k}(b), M^{\mathcal{B}}_{n,k}(b) \right) < \frac{\epsilon}{4}$.

\begin{claim}\label{cla1} There is a partition $\mathcal{Q}$ of $X$ indexed by $r \times 2$ such that for every $\theta: J \to r \times 2$ with $J \subseteq F_0 F$ we have $|(\mu \times \nu)(P^b_{\theta}) - \mu(Q^c_{\theta})| < \delta$.  \end{claim}

Suppose the claim holds. Regard $\mathbb{F}_N$ as acting on $\bigcup_{J \subseteq \mathbb{F}_N} \{\theta: J \to 2 \times r \}$ by shift, $\gamma \cdot \theta(\gamma') = \theta(\gamma^{-1}\gamma')$. Thus the domain $\mathrm{dom}(\gamma \cdot \theta) = \gamma \mathrm{dom}(\theta)$. Then for any $\theta,\kappa: F \to 2 \times r$ and $\gamma \in F_0$ we have \[ \gamma^b P^b_{\theta} \cap P^b_{\kappa} = \begin{cases} P^b_{\gamma \cdot \theta \cup \kappa} &\mbox{ if }\gamma \cdot \theta \mbox{ and } \kappa \mbox{ are compatible},\\ \emptyset &\mbox{ if not.} \end{cases} \] and similarly \[ \gamma^c Q^c_{\theta} \cap Q^c_{\kappa} = \begin{cases} Q^c_{\gamma \cdot \theta \cup \kappa} &\mbox{ if }\gamma \cdot \theta \mbox{ and } \kappa \mbox{ are compatible},\\ \emptyset &\mbox{ if not.} \end{cases} \]

Therefore the claim gives $|(\mu \times \nu)(\gamma^b P^b_{\theta} \cap P^b_{\kappa}) - \mu(\gamma^c Q^c_{\theta} \cap Q^c_{\kappa})| <  \delta$ for all $\theta,\kappa: F \to r \times 2$. So if $\mathcal{B} = \{B_1,\ldots,B_k\}$ is a partition such that $B_i = \bigsqcup_{s=1}^t P^b_{\theta_i(s)}$ for functions $\theta_i(s):F \to r \times 2$ and we let $C_i = \bigsqcup_{s=1}^t Q^c_{\theta_i(s)}$ then we have \begin{align*} |(\mu \times \nu)(\gamma^b B_i \cap B_j) - \mu(\gamma^c C_i \cap C_j)|  &= \left \vert (\mu \times \nu) \left( \bigsqcup_{s,s' = 1}^t \gamma^b P^b_{\theta_i(s)} \cap P^b_{\theta_j(s')} \right) - \mu\left( \bigsqcup_{s,s' =1}^t \gamma^c Q^c_{\theta_i(s)} \cap Q^c_{\theta_j(s')} \right) \right \vert \\ & \leq t^2 \delta \leq (2r)^{2|F|} \delta < \frac{\epsilon}{4}, \end{align*}
since $t \leq (2r)^{|F|}$. Taking $\mathcal{C} = (C_i)_{i=1}^k$ we get $\rho \left( M^\mathcal{B}_{n,k}(b),M^\mathcal{C}_{n,k}(c) \right) < \frac{\epsilon}{4}$, which implies the theorem.\\
\\
It remains to show Claim \ref{cla1}. This part of the argument follows the proof of Theorem 1 in \cite{AW} and the extensions of these ideas developed in \cite{RTD}. Let $G = F_0 F$. Assume without loss of generality that $G$ is closed under taking inverses. Note that it suffices to prove the claim for $\theta$ defined on all of $G$. In order to find $\mathcal{Q}$ we will find a partition $\mathcal{T} = \{ T_1,T_2\}$ and set $Q_{i,j} = R_i \cap T_j$ for $1 \leq i \leq r$, $1 \leq j \leq 2$. Thus we are looking for $\mathcal{T} = \{T_1,T_2\}$ such that for all $(\tau,\sigma)$ with $\sigma:G \to r$ and $\tau: G \to 2$ we have \[ |(\mu \times \nu)(R^c_{\sigma} \times S^s_{\tau}) - \mu(R^c_{\sigma} \cap T^c_\tau)| < \delta. \] Note that $\nu(S^s_\tau) = 2^{-|G|}$ for any such $\tau$ so we are looking for $\mathcal{T}$ such that $\left \vert 2^{-|G|}\mu(R^c_{\sigma}) - \mu(R^c_{\sigma} \cap T^c_{\tau}) \right \vert < \delta$. The idea is that a random $\mathcal{T}$ should have this property.\\
\\
Without loss of generality we may assume $X$ is a compact metric space with a compatible metric $p$. For $\eta > 0$ let \[ D_\eta = \{x \in X: \mbox{ for all } \gamma, \gamma' \in G, \gamma_1 \neq \gamma_2 \mbox{ implies } p(\gamma_1^c x, \gamma_2^c x) > \eta \}\] and \[ E_\eta = \{(x,x') \in D_\eta^2: \mbox{ for all } \gamma_1,\gamma_2 \in G, p(\gamma_1^c x, \gamma_2^c x') > \eta\}. \]

\begin{lemma} There is $\eta >0$ such that $\mu(D_\eta) > 1 - \frac{\delta^2}{16 (2r)^{|F|^2}}$ and $\mu^2(X^2 \setminus E_\eta) < \frac{\delta^2}{16 (2r)^{2|F|}}$. \end{lemma}

\begin{proof} Clearly if $\eta_1 < \eta_2$ then $D_{\eta_2} \subseteq D_{\eta_1}$. We have $X \setminus \bigcup_{\eta > 0} D_\eta = \{ x \in X:$ for some $\gamma_1 \neq \gamma_2 \in G$, $\gamma_1^c x = \gamma_2^c x \}$. Now since $a$ is free, if $\gamma_1^c x = \gamma_2^c x$ then we must have $\gamma_i^c x \neq \gamma_i^a x$ for some $i \in \{1,2\}$. Each $\gamma \in G$ is a product $f_1 f_2$ for $f_1 \in F_0$ and $f_2 \in F$, thus for any $\gamma \in G$ we have \[ d_u \left( \gamma^c,\gamma^a \right) < d_u(f_1^a,f_1^c) + d_u(f_2^a,f_2^c) < \frac{\delta^2}{16 |F|^2 (2r)^{|F|^2}} \] since $f_i \in F$. Therefore \[ \mu ( \{ x: \mbox{ for some }\gamma \in G, \gamma^c x \neq \gamma^a x \}) < |G| \frac{\delta^2}{16 |F|^2 (2r)^{|F|^2}} < \frac{\delta^2}{16 (2r)^{2|F|}}.\] and hence $\mu \left( X \setminus \bigcup_{\eta > 0} D_\eta \right) < \frac{\delta^2}{16 (2r)^{|F|^2}}$. So we can find $\eta = \eta_0$ such that $D_{\eta_0}$ satisfies the lemma. Now for any $\eta > 0$, \begin{align*} D_{\eta_0}^2 \setminus \bigcup_{\eta > 0} E_\eta &= \{(x,x') \in D_{\eta_0}^2: \mbox{ for all } \eta > 0 \mbox{ there exist }\gamma_1,\gamma_2 \in G \mbox{ such that } p(\gamma_1 x,\gamma_2 x') < \eta \}\\ &= \{(x,x') \in D_{\eta_0}^2: \mbox{ there exist }\gamma_1,\gamma_2 \in G \mbox{ such that } \gamma_1 x = \gamma_2 x' \}. \end{align*}
For a fixed $x$, $\{(x,x') \in D_{\eta_0}^2: \mbox{ there exist }\gamma_1,\gamma_2 \in G \mbox{ such that } \gamma_1 x = \gamma_2 x' \}$ is finite so $\mu \left( D_{\eta_0}^2 \setminus \bigcup_{\eta > 0} E_\eta \right)$ has measure $0$ by Fubini and hence we have the lemma for $E_\eta$.  \end{proof}   

Let $\mathcal{Y} = \{Y_1,\ldots,Y_m\}$ be a partition of $X$ into pieces with diameter $< \frac{\eta}{4}$. For $x \in X$ let $Y(x)$ be the unique $l \leq m$ such that $x \in Y_i$. Let $\kappa$ be the uniform (= product) probability measure on $2^m$ and for each $\omega \in 2^m$ define a partition $Z(\omega) = \{Z^\omega_1,Z^\omega_2\}$ by letting $x \in Z^\omega_i$ if and only if $\omega(Y(x)) = i$. Thus we have a random variable $Z: (2^m,\kappa) \to \mathrm{MALG}(X,\mu)^2 $ given by $\omega \mapsto Z(\omega)$. Fix now $\tau:G \to 2$ and an arbitrary subset $A \subseteq X$. We compute the expected value of $\mu(Z(\omega)_\tau \cap A)$. Let $\chi_B$ be the characteristic function of $B$. 

\begin{align} \mathbb{E}[\mu(Z_\tau \cap A)] &= \int_{2^m} \mu(Z(\omega)_\tau \cap A) d \kappa(\omega) \nonumber \\  & = \int_{2^m} \int_X \chi_{Z(\omega)_\tau \cap A}(x) d \mu(x) d \kappa^m(\omega) \nonumber \\ &= \int_A \int_{2^m} \chi_{Z(\omega)_\tau}(x) d \kappa(\omega) d \mu(x) \nonumber \\ & = \int_{D_\eta \cap A} \int_{2^m} \chi_{Z(\omega)_\tau}(x) d \kappa(\omega) d \mu(x) + \int_{A \setminus D_\eta} \int_{2^m} \chi_{Z(\omega)_\tau}(x) d \kappa(\omega) d \mu(x) \end{align}

Now if $x \in D_\eta$ then for all $\gamma_1 \neq \gamma_2 \in G$ we have $p(\gamma_1^c x, \gamma_2^c x) \geq \eta$ so that $Y(\gamma_1^c x) \neq Y(\gamma_2^c x)$ and hence the events $\omega(Y(\gamma_1^c x)) = i$ and $\omega(Y(\gamma_2^c x)) = j$ are independent. We have $x \in \gamma^c Z(\omega)_{\tau(\gamma)}$ if and only if $\omega(Y((\gamma^{-1})^c x)) = \tau(\gamma)$, so if $x \in D_\eta$ and $\gamma_1 \neq \gamma_2 \in G$ the events $x \in \gamma^c Z(\omega)_{\tau(\gamma_1)}$ and $x \in \gamma^c Z(\omega)_{\tau(\gamma_2)}$ are independent. So for $x \in D_\eta$,

\begin{align} \int_{2^m} \chi_{Z(\omega)_\tau}(x) d \kappa(\omega) &= \kappa (\{\omega: x \in \gamma^c Z(\omega)_{\tau(\gamma)} \mbox{ for all } \gamma \in G \}) \nonumber \\ & = \prod_{\gamma \in G}\kappa \left ( \left\{ \omega: \omega(Y((\gamma^{-1})^cx)) = \tau(\gamma) \right\} \right) = 2^{-|G|}  \end{align}

Since $\mu(X \setminus D_\eta) < \frac{\delta^2}{16 (2r)^{|F|^2}}$, we have $2^{-|G|}\left(\mu(A) - \frac{\delta^2}{16 (2r)^{|F|^2}} \right) \leq (6) \leq  2^{-|G|}\mu(A)+ \frac{\delta^2}{16 (2r)^{|F|^2}}$ and thus $\left \vert \mathbb{E}[\mu(Z_\tau \cap A)] - \mu(A)2^{-|G|} \right \vert < \frac{\delta^2}{16 (2r)^{|F|^2}}.$ We now compute the second moment of $\mu(Z_\tau \cap A)$, in order to estimate its variance.

\begin{align} \mathbb{E}\left[\mu(Z_\tau \cap A)^2 \right] & = \int_{2^m} \mu(Z_\tau(\omega) \cap A)^2 d \kappa (\omega) \nonumber \\ & = \int_{2^m} \left( \int_A \chi_{Z_\tau(\omega)}(x) d \mu(x) \right)^2 d \kappa(\omega) \nonumber \\ & = \int_{2^m} \int_{A^2} \chi_{Z_\tau(\omega)}(x_1) \chi_{Z_\tau(\omega)}(x_2) d \mu^2(x_1,x_2) d \kappa(\omega)\nonumber \\ & = \int_{A^2} \int_{2^m} \chi_{Z_\tau(\omega)}(x_1) \chi_{Z_{\tau}(\omega)}(x_2) d \kappa (\omega) d \mu^2(x_1,x_2) \nonumber \\ & =  \int_{A^2 \cap E_\eta} \int_{2^m} \chi_{Z_\tau(\omega)}(x_1) \chi_{Z_{\tau}(\omega)}(x_2) d \kappa (\omega) d \mu^2(x_1,x_2) \nonumber \\ & \hspace{1 in}+ \int_{A^2 \setminus E_\eta} \int_{2^m} \chi_{Z_\tau(\omega)}(x_1) \chi_{Z_{\tau}(\omega)}(x_2) d \kappa (\omega) d \mu^2(x_1,x_2) \end{align}

Now if $(x_1,x_2) \in E_\eta$ then for any pair $\gamma_1,\gamma_2 \in G$ we have $p(\gamma^c_1 x_1,\gamma^c_2 x_2) > \eta$ so that $Y(\gamma^c_1 x_1) \neq Y(\gamma^c_2 x_2)$ and thus for a fixed pair $(x_1,x_2)$ the events $\omega(Y(\gamma^{-1})^c x_1) = \tau(\gamma)$ for all $\gamma \in G$ and $\omega(Y(\gamma^{-1})^c x_2) = \tau(\gamma)$ for all $\gamma \in G$ are independent. Hence for a fixed $(x_1,x_2) \in E_\eta$ we have

\begin{align*} \int_{2^m} \chi_{Z_\tau(\omega)}(x_1) \chi_{Z_\tau(\omega)}(x_2) d \kappa(\omega) & = \kappa(\{ \omega: x_1 \in \gamma^c Z(\omega)_{\tau(\gamma)} \mbox{ and } x_2 \in \gamma^c Z(\omega)_{\tau(\gamma)} \mbox{ for all } \gamma \in G \}) \\ & = \kappa(\{ \omega: \omega(Y((\gamma^{-1})^cx_1) = \tau(\gamma) \mbox{ and } \omega(Y((\gamma^{-1})^c) x_2) = \tau(\gamma) \mbox{ for all } \gamma \in G \}) \\ & =\kappa \left ( \left\{ \omega: \omega(Y((\gamma^{-1})^cx_1)) = \tau(\gamma) \mbox{ for all }\gamma \in G \right\} \right) \\ & \hspace{1 in} \cdot \kappa \left ( \left\{ \omega: \omega(Y((\gamma^{-1})^cx_2)) = \tau(\gamma) \mbox{ for all } \gamma \in G \right\} \right) \\ & = 2^{-2|G|} \end{align*}
 by $(7)$ and the fact that $E_\eta \subseteq D_\eta^2$. Since $\mu^2(A \setminus E_\eta) < \frac{\delta^2}{16 (2r)^{|F|^2}}$ we see $\left(\mu(A)^2 - \frac{\delta^2}{16 (2r)^{|F|^2}} \right) 2^{-2|G|} \leq (8) \leq 2^{-2|G|} \mu(A)^2 + \frac{\delta^2}{16 (2r)^{|F|^2}}$ and hence $\left \vert \mathbb{E}[\mu(Z_\tau \cap A)^2] - \mu(A)^2 2^{-2|G|} \right \vert < \frac{\delta^2}{16 (2r)^{|F|^2}}$. Therefore \begin{align*} \mathrm{Var}(\mu(Z_\tau \cap A)) &= \mathbb{E}[\mu(Z_\tau \cap A)^2] - \mathbb{E}[\mu(Z_\tau \cap A)]^2 \\ & \leq \left \vert \mathbb{E}[\mu(Z_\tau \cap A)^2] - \mu(A)^2 2^{-2|G|} \right \vert + \mu(A)^2 2^{-2|G|} - \left ( - \left \vert \mathbb{E}[\mu(Z_\tau \cap A)] - \mu(A)2^{-|G|} \right \vert + \mu(A)2^{-|G|} \right)^2 \\ & \leq \frac{\delta^2}{16(2r)^{|F|^2}} + \mu(A)^2 2^{-2|G|} - \left( - \frac{\delta^2}{16(2r)^{|F|^2}} + \mu(A)2^{-|G|} \right) ^2 \\ & = \frac{\delta^2}{16(2r)^{|F|^2}} - \frac{\delta^4}{(16(2r)^{|F|^2})^2} + 2 \mu(A)2^{-|G|} \frac{\delta^2}{16 ( 2r)^{|F|^2}} \leq \frac{\delta^2}{8 (2r)^{|F|^2}} . \end{align*}

Therefore Chebyshev's inequality for $\mu(Z_\tau \cap A)$ gives

\begin{align*}\kappa \left( \left\{ \omega: |\mu(Z_\tau(\omega) \cap A) - \mathbb{E}[\mu(Z_\tau \cap A)] | \geq \frac{\delta}{2} \right \} \right) &\leq \frac{\mathrm{Var}(\mu(Z_\tau \cap A))}{\left(\frac{\delta}{2}\right)^2} \\ & \leq \frac{1}{2 (2r)^{|F|^2} }   \end{align*}

Now since $\left \vert \mathbb{E}[\mu(Z_\tau \cap A)] - \mu(A)2^{-|G|} \right \vert < \frac{\delta}{2}$ we have \[\kappa \left( \left\{ \omega: \left \vert \mu(Z_\tau(\omega) \cap A) - \mu(A)2^{-|G|} \right \vert \geq \delta \right \} \right) \leq  \frac{1}{2 (2r)^{|F|^2} }.\]

Since this is true for each $\tau \in 2^G$ we have 

 \[\kappa \left( \left\{ \omega:\left \vert \mu(Z_\tau(\omega) \cap A) - \mu(A)2^{-|G|}\right \vert| \geq \delta \mbox{ for some } \tau: G \to 2  \right \} \right ) \leq  \frac{1}{2 r^{|F|^2} }.\]

Finally, letting $A$ range over the sets $R_\sigma$ for $\sigma \in r^G$ we get 

 \[\kappa \left ( \left\{ \omega: \left \vert \mu(Z_\tau(\omega) \cap R^c_\sigma) - \mu(R^c_\sigma)2^{-|G|} \right \vert \geq \delta \mbox{ for some } \tau: G \to 2 \mbox{ and } \sigma: G \to r \right \} \right) \leq  \frac{1}{2} .\]

Then any member of the nonempty complement of \[ \left\{ \omega: \left \vert \mu(Z_\tau(\omega) \cap R^c_\sigma) - \mu(R^c_\sigma)2^{-|G|} \right \vert \geq \delta \mbox{ for some } \tau: G \to 2 \mbox{ and } \sigma: G \to r \right \} \] works as $\mathcal{T}$. This completes the proof of Theorem \ref{thm5}.

\end{proof}

We note that the proof of Theorem \ref{thm5} goes through for any group $\Gamma$ such that an arbitrary free action can be approximated in the uniform topology by ergodic actions - for example any group of the form $\mathbb{Z} * H$. Such an approximation is impossible if $\Gamma$ has property $\mathrm{(T)}$, and in this case the extreme points of $\mathrm{FR}_{\sim_s}(\Gamma,X,\mu)$ are closed. Therefore the following question is natural.

\begin{question} Let $\Gamma$ be a group without property $\mathrm{(T)}$. Can every free action of $\Gamma$ be approximated in the uniform topology of $A(\Gamma,X,\mu)$ by ergodic actions? \end{question}

\section{The space of stable weak equivalence classes.}

$\mathrm{A}_{\sim_s}(\Gamma,X,\mu)$ can be given the structure of a weak convex space in exactly the same way as $\mathrm{A}_\sim(\Gamma,X,\mu)$. Moreover, it is clear that for any $a \in \mathrm{A}(\Gamma,X,\mu)$ and $t \in [0,1]$ we have $a \sim_s ta + (1-t)a$, so $\mathrm{A}_{\sim_s}(\Gamma,X,\mu)$ is in fact a convex space. Recall that the metric $d_s$ on $\mathrm{A}_{\sim_s}(\Gamma,X,\mu)$ is defined by $d_s(a,b) = d(a \times \iota,b \times \iota)$ where $d$ is the metric on $\mathrm{A}_\sim(\Gamma,X,\mu)$.

\begin{proposition}\label{prop12} For any $a,b,c \in \mathrm{A}(\Gamma,X,\mu)$ and $t \in [0,1]$, we have $d_s(ta + (1-t)c,tb + (1-t)c) \leq td_s(a,b)$. \end{proposition}

It is clear that $(ta + (1-t)c) \times \iota \sim t(a \times \iota) + (1-t)(c \times \iota)$, so it suffices to show the following.

\begin{proposition}\label{prop5} For any $a,b,c \in \mathrm{A}(\Gamma,X,\mu)$ and $t \in [0,1]$ we have $d(ta + (1-t)c,tb  + (1-t)c) \leq t d(a,b)$. \end{proposition}

\begin{proof} Fix $n,k$ and write $C(a) = C_{n,k}(a)$ in order to show that $d_H(C(ta + (1-t)c), C(tb + (1-t)c)) \leq t d_H(C(a),C(b))$. Fix $\epsilon > 0$. Let $\mathcal{P} = (P_i)_{i=1}^n$ be a partition of $ X_1 \sqcup X_2$ where $X_1$ and $X_2$ are disjoint copies of $X$. Let $P_i^l = P_i \cap X_l$ for $l \in \{1,2\}$. Find a partition $\mathcal{Q} = (Q_i)_{i=1}^n$ such that for $i,j \leq n$ and $p \leq k$ we have \[ | \mu(\gamma_p^a P^1_i \cap P^1_j) - \mu(\gamma_p^b Q_i \cap Q_j)| < d_H(C(a),C(b)) + \epsilon.\] Then if we take $Q_i' = Q_i \sqcup P^2_i$ for all $i,j \leq n$, \begin{align*} | (t\mu + (1-t) \mu)&( \gamma_p^{ta + (1-t)c} P_i \cap P_j) - (t \mu + (1-t) \mu)(\gamma_p^{tb + (1-t)c} Q_i' \cap Q_j')| \\ & = | t \mu( \gamma_p^a P^1_i \cap P^1_j) + (1-t) \mu( \gamma_p^c P^2_i \cap P^2_j) - t\mu( \gamma_p^b Q_i \cap Q_j) - (1-t) \mu(\gamma_p^c P^2_i \cap P^2_j) | \\ & = | t \mu( \gamma_p^a P^1_i \cap P^1_j) +t\mu( \gamma_p^b Q_i \cap Q_j)| \leq t (d_H(C(a),C(b))  + \epsilon). \end{align*}

 \end{proof}

Theorem \ref{thm3} now follows from Proposition \ref{prop12} and Corollary 12 in \cite{CFr}. Tucker-Drob and Bowen have obtained the next result independently of the author.

\begin{proposition} The extreme points of $\mathrm{A}_{\sim_s}(\Gamma,X,\mu)$ are precisely those stable weak equivalence classes which contain an ergodic action. \end{proposition}

\begin{proof} Suppose that $a$ is ergodic and we have $a \sim_s tb + (1-t)c$ for $t \in (0,1)$. Therefore $a \prec \iota \times (tb + (1-t)c) \sim t( b \times \iota) + (1-t)(c \times \iota)$. Since $a$ is ergodic, Theorem 3.11 in \cite{RTD} implies that $a \prec b$ and $a \prec c$. Suppose toward a contradiction that $b \nprec_s c$, so that for some $n,k$ we have $C_{n,k}(b) \nsubseteq \cch( C_{n,k}(c) )$. Fixing $n,k$ write $C(d)$ for $C_{n,k}(d)$. Let $\alpha = \sup_{x \in C(b)} p(x, \cch(C(c)) )$ where $p$ is the metric on $[0,1]^{n \times k \times k}$. Choose $x_0 \in C(b)$ so that $p(x_0, \cch(C(c))) = \alpha$. Choose $y_0 \in \cch(C(c))$ so that $p(x_0,y_0) = \alpha$. Consider the point $tx_0 + (1-t) y_0 \in \cch(C(tb + (1-t)c))$. It is easy to see that \[ p(tx + (1-t) z, ty  + (1-t) z) \leq tp(y,z) \] for any $x,y,z$ so we have

 \begin{align*} p(t x_0 + (1-t) y_0, x_0) & = p(t x_0 + (1-t)y_0, tx_0 + (1-t) x_0) \\ & \leq (1-t) p(x_0,y_0) < \alpha \end{align*}

since $0 < t$. Since $\alpha = \inf_{y \in \cch(C(c))} p(x_0,y)$ we see that $tx_0 + (1-t) y_0 \notin \cch(C(c))$ and hence $\cch(C(tb+(1-t)c) \nsubseteq \cch(C(c))$. Since for any two actions $d,e$ we have $d \prec_s e$ if and only if $\cch(C_{n,k}(d)) \subseteq \cch(C_{n,k}(e))$ for all $n,k$ this implies that $tb + (1-t)c \nprec_s c$. But $tb + (1-t)c \prec_s a \prec c$ by hypothesis, so we have a contradiction and we conclude $b \prec_s c$. A symmetric argument shows $c \prec_s b$, so $b \sim_s c$. Since $\mathrm{A}_{\sim_s}(\Gamma,X,\mu)$ obeys $(2)$ of Definition \ref{def1}, we get that $a \sim_s b \sim_s c$. Therefore if a stable weak equivalence class contains an ergodic action, it is an extreme point of $A_{\sim_s}(\Gamma,X,\mu)$. On the other hand, an argument identical to the proof of Theorem \ref{thm4} shows that if the stable weak equivalence class of an action $a$ is an extreme point of $A_{\sim_s}(\Gamma,X,\mu)$ then if we write $a = \int_Z a_z d \eta(z)$ then there is an ergodic action $b$ such that $a_z \sim_s b$ for all $z \in Z$. Thus $a \sim_s b \times \iota \sim_s b$ and we see that $a$ is stably weakly equivalent to an ergodic action. \end{proof}

\bibliographystyle{plain}
\bibliography{bibliography}

Department of Mathematics\\
California Institute of Technology\\
Pasadena CA, 91125\\
\texttt{pjburton@caltech.edu}

\end{document}